[1994/12/01]

\documentclass{amsart}
\usepackage{amssymb}
\usepackage{amsmath,amssymb,amsfonts,amsthm,graphics,
latexsym, amscd, amsfonts, epsfig, eepic,epic}
\usepackage{mathrsfs}
\usepackage{algorithm}
\usepackage{algorithmic}
\usepackage{color}

\usepackage{eucal}
\usepackage{eufrak}
\usepackage[all]{xypic}
\usepackage{xspace}

\newtheorem{thm}{Theorem}[section]
\newtheorem{lem}[thm]{Lemma}
\newtheorem{prop}[thm]{Proposition}
\newtheorem{cor}[thm]{Corollary}

\newtheorem{rmk}[thm]{Remark}

\theoremstyle{definition}
\newtheorem{defn}{Definition}[section]

\theoremstyle{remark}

\makeatletter \@addtoreset{equation}{section}

\newcommand{\thmref}[1]{Theorem~\ref{#1}}

\def\a{\alpha}
\def\b{\beta}

\def\g{\gamma}
\def\s{\sigma}
\def\t{\tau}

\def\l{\lambda}

\def\p{\partial}
\def\vphi{\varphi}

\def\g{\mathfrak g}

\def\and{\quad{\rm and}\quad}

\let\lra=\longrightarrow

\def\mapright{\xrightarrow}
\def\mapright\#1{\,\smash{\mathop{\lra}\limits^{\#1}}\,}

\def\om{\omega}
\def\Om{\Omega}

\def\tri{\triangle}

\newcommand{\as}{\underline{S}}

\numberwithin{equation}{section}

\DeclareMathOperator{\trivial}{trivial}
\DeclareMathOperator{\dist}{dist}

\makeatletter

\newcommand{\Rmnum}[1]{\expandafter\@slowromancap\romannumeral #1@}

   \def\sideremark#1{\ifvmode\leavevmode\fi\vadjust{
   \vbox to0pt{\hbox to 0pt{\hskip\hsize\hskip1em
   \vbox{\hsize3cm\tiny\raggedright\pretolerance10000
   \noindent #1\hfill}\hss}\vbox to8pt{\vfil}\vss}}}
\title[The Dirichlet and the weighted metrics]{ The Dirichlet and the weighted metrics for the space of K\"ahler metrics}

\author [Simone Calamai]{Simone Calamai}
  \address{Scuola Normale Superiore di Pisa - Piazza dei Cavalieri, 7 - 56126 Pisa, Italy }
  \email{simocala@gmail.com}
\author [Kai Zheng]{Kai Zheng}
  \address{Institut Fourier, Universit\'{e} Joseph Fourier (Grenoble I), UMR 5582 CNRS-UJF, BP 74, 38402 Saint-Martin-d'H\`{e}res, France}
  \email{kai.zheng@ujf-grenoble.fr}
\thanks{The first author is partially supported by INdAM grants. The second author is partially supported by ANR project "Flots et Op\'erateurs
G\'eom\'etriques" (ANR-07-BLAN-0251-01).}
\begin{document}
\maketitle
\begin{abstract}
In this work we study the intrinsic geometry of the space of K\"ahler metrics under various Riemannian metrics.
The first part is on the \emph{Dirichlet metric}.
We motivate its study, we compute its curvature, and we make links with the
Calabi metric, the K-energy, the degenerate complex Hessian equation.
The second part is on the weighted metrics, for which we investigate as well
their geometric properties.
\end{abstract}
\tableofcontents
\section{Introduction}
Let $(M,\, \omega)$ be a compact closed K\"ahler manifold of complex dimension $n$ and let $\Om$ be the K\"ahler class
of $\omega$;
namely, $\omega \in \Om \in H^2_{dR}(M, \, \mathbb{R})$.
The space of K\"ahler metrics $\mathcal{H}$ corresponding to the K\"ahler class $\Om$ is the
subset of $\Om$ containing all K\"ahler metrics.
We will explain in next section that the space of K\"ahler metrics
is in  one to one correspondence both to the space of K\"ahler potentials \eqref{potential}
under the normalization condition \eqref{ding},
and to the space of volume conformal factors \eqref{conformal}
and also to the space of K\"ahler one-forms \eqref{one form}.

In \cite{Calabi},
Calabi proposed a variational problem on  minimizing the Calabi energy in the space of K\"ahler metrics.
Any minimizer is called an extremal metric.
Calabi's extremal metrics generalize the constant scalar curvature K\"ahler metrics;
as proved by Calabi in \cite[page 99, Theorem 4)]{Calabi2}, whenever a space of  K\"ahler metrics $\mathcal{H}$ admits
an extremal K\"ahler metric, then the vanishing of the Calabi-Futaki invariant is equivalent to the
existence of a constant scalar curvature K\"ahler metric in $\mathcal{H}$.
The underlying geometric structure of the moduli space $\mathcal{H}$ has been extensively studied in recent years.
In this paper, we study the intrinsic geometry of the space of K\"ahler metrics under various Riemannian metrics.

The most studied Riemmanian metric in $\mathcal{H}$ is the \emph{Mabuchi metric} which is defined in Mabuchi \cite{Mab},
Donaldson \cite{Don} and Semmes \cite{Sem} independently. Under this metric,
$\mathcal{H}$ becomes a non-positively curved infinite-dimensional symmetric space.
Semmes \cite{Sem} pointed out that the geodesic equation in $\mathcal{H}$
is a homogeneous complex Monge-Amp\`ere equation. Donaldson \cite{Don} conjectured
that $\mathcal{H}$ endowed with the Mabuchi metric is geodesically convex and is a metric space
and pointed out the intensive relation between the geodesics of $\mathcal{H}$ and the existence (through geodesic stability),
uniqueness (through convexity along the geodesics) of the extremal metrics. We refer readers to our paper \cite{CalamaiZheng} for complete references. In this paper \cite{CalamaiZheng}, we find geometric conditions on the Dirichlet boundary values with less regularity which assure the existence and uniqueness of $C^{1,1}$ geodesic segments.


A second Riemannian structure on the space of volume conformal factors \eqref{conformal} was hinted in
\cite{Calabi} and recently developed in detail by the first author in \cite{Calamai};
it is called  \emph{the Calabi metric} \eqref{calabimetric}. It enjoys nice properties,
such as a positive constant sectional curvature (while the Mabuchi metric has
non-constant non-positive sectional curvature), and it induces a metric space structure on
$\mathcal{H}$ (as well as the Mabuchi metric does).
Moreover, the Calabi metric endows $\mathcal{H}$ with a richer geometry than the Mabuchi one;
both for the Cauchy and the Dirichlet problems, there are explicit, real analytic solutions.
The work done in \cite{Calamai} immediately lead a good deal of advances on the Calabi metric.
In \cite{ClaRub}, it is shown that the restriction of the Ebin metric \cite{Ebin}, defined on the space $\mathcal{M}$ of
all Riemannan metrics, to the space $\mathcal{H}$ is precisely the Calabi metric,
and the metric completion of $\mathcal{H}$ endowed with the Calabi metric is explicitly computed;
in \cite{ClaRub2} Calabi metric was generalized in order to make the space $\mathcal{H}$
as totally geodesic as possible in $\mathcal{M}$ endowed with the Ebin metric.

In Section $3$, we introduce a third way to describe the space of K\"ahler metrics, which is
from the point of view of differential forms.
In the space of K\"ahler one-forms \eqref{one form}, we define the Dirichlet metric as a canonical $L^2$-inner
product and see how it induces an almost complex structure naturally (cf. Remark \ref{rmk: almost complex structure}).

A motivation for the study of this metric comes from \cite{ChenZheng};
there, Chen and the second author showed that the pseudo-Calabi flow is the gradient
flow of the $K$-energy when $\mathcal{H}$ is endowed precisely with the Dirichlet metric.
In the same paper, they also developed a geometric method based on $\mathcal{H}$
and proved the stability of the the pseudo-Calabi flow near a cscK metric.
In this paper, we prove its counterpart in the variational problem:
the $K$-energy is convex at a point which corresponds to a constant
scalar curvature K\"ahler metric (cf. Proposition \ref{prop: convexity at a cscK metric}).
We also present a formulation for the geodesic equation and we present the conjecture
that  the regularity of geodesics (if they exist)
is greater than the $C^{1,1}$ regularity of Chen's geodesics
and smaller that the $C^{\omega}$ regularity of Calabi's geodesics.

Now let us focus on the "shape" of the space $\mathcal{H}$.
It is well known by the work of Mabuchi \cite{Mab}, that
the Mabuchi metric has  non-positive sectional curvature.
On the other hand the Calabi metric, as shown in \cite{Calamai},
has positive constant sectional curvature. When $M$ is a Riemann surface,
the space $\mathcal{H}$, endowed with the Dirichlet metric, becomes flat.
For manifolds of higher dimension, it is not known whether $\mathcal{H}$ endowed with
the Dirichlet metric is flat. Calabi suggests that the sectional curvature of the
Dirichlet metric should be between the sectional curvature of the Mabuchi metric and of the Calabi metric.
In this paper, we first find out the explicit formula of the sectional curvature of the Dirichlet metric.
Then we prove that for a two plane
spanned by tangent vectors $\psi_1 , \psi_2 \in T_\phi \mathcal{H}$, the sectional curvature of the Dirichlet metric
is bounded above and below by a constant $K$ which depends only on the point
$\phi$ and on one generator $\psi_1$ (cf. Theorem \ref{thmsectionalcurvature}).

We find out the sufficient conditions to determine the signature of the sectional curvature, they are the degenerate complex Hessian equations
(see Remark \ref{rmk: complex hessian equation}).
This leads us to conjecture that unlike the case when $M$ is a Riemann surface, the sectional curvature is not flat.

In the second part of the paper, we introduce a family of Riemannian metrics
on the space $\mathcal{C}$ of volume conformal factors, and we call them weighted metric.
In particular, the weighted metric whose weight function $\chi$ is $\chi(x)=x^{-1}$, is precisely the Calabi metric.
For each of the weighted metrics, we compute its sectional curvature (Theorem \ref{thm: formula of sectional curvature})
and its geodesic equation (Proposition \ref{prop: geodesic equation}).
We call the metric corresponding to the weight function $\chi(x) = 1$ the constant weight metric.
The constant weight metric enjoys several nice properties.
\begin{itemize}
\item It is a flat Riemannian metric (cf. Corollary \ref{cor: sectional curvature of constante weight metric});
\item Its geodesics are smooth lines (cf. Theorem \ref{thm: geodesics for the constant weight metric});
\item The $K$-energy functional  is convex along the geodesic in the first Chern class when $C_1\leq0$.
\end{itemize}
When $C_1\leq0$, these properties enable us to give an alternative proof of the uniqueness
of the K\"ahler-Einstein metrics (cf. \thmref{unike}). We also find a new functional and we prove that its gradient flow also convergence to the K\"ahler-Einstein metrics (cf. \thmref{exke}).


\section{Notations and definitions}
With $M, \omega$ and $\Om$ as in the Introduction, recall that the space of K\"ahler metrics is
\begin{align}
\mathcal H=\{\om\vert\om\text{ is a K\"ahler metric in }\Om\} \; .
\end{align}
In the given K\"ahler class $\Om$ the total volume of $M$ is a fixed constant that only depends on K\"ahler class,
and we call it $V$ in the remainder.

Choose a K\"ahler metric $\omega_1$ in the K\"ahler class;
according to the $\p\bar\p$-lemma, for any smooth K\"ahler metric,
there exists a unique (up to a constant) smooth function $\vphi$ (called K\"ahler potential)
such that $\omega_1=\om+i\p\bar\p\vphi$.
So in order to assure the uniqueness of the K\"ahler potential,
we require a normalization condition by means of the following functional on $\mathcal{H}$
\begin{align}\label{ding}
D_\om(\vphi)=\sum_{i=0}^n\frac{n!}{(i+1)!(n-i)!}
\int_{M}\vphi \, \om^{n-i}\wedge(\p\bar\p\vphi)^i \; .
\end{align}
Since its first variation is $\int_M\dot\vphi\om_\vphi^n$, we see that the Monge-Amp\`{e}re
operator is the Euler-Lagrange operator of $D_\om$.

Then the space of K\"ahler metrics is one to one correspondence to the space of
K\"ahler potentials which is defined to be the set containing all the K\"ahler
potentials with the following normalization condition, i.e.
\begin{align}\label{potential}
\mathcal{H}=\{\vphi\in C^{\infty}(M , \, \mathbb{R}) \; \vert \; \om+i\p\bar\p\vphi>0, D_\om(\vphi)=0\} \; .
\end{align}
Whenever we have an element $\alpha \in C^{1}([0,1]\times M ,\, \mathbb{R})$
such that $\alpha (t, \cdot ) \in \mathcal{H}$, for any $t\in [0,1]$, then
we call it a \emph{differentiable curve},
and we write it simply as $\alpha : [0,1] \rightarrow \mathcal{H}$ when no confusion arises.
By means of the characterization \eqref{potential} and of differentiable curves,
the tangent space of $\mathcal{H}$ at a point $\phi$ can be expressed as
\begin{align*}
 T_\phi \mathcal{H} = \left\{ \psi \in C^{\infty}(M , \, \mathbb{R}) \, \left| \, \int_M \psi \frac{\omega_\phi^n}{n!} =0
 \right. \right\}\; .
\end{align*}
In the remainder of the paper, we will need to consider differentiable curves of tangent vectors, in the following sense.
Given a differentiable curve $\alpha$ as before, we will consider a map
$\beta \in C^{1}([0,1]\times M ,\, \mathbb{R})$ such that, for any $t\in [0,1]$, then
$\beta (t, \cdot ) \in T_{\alpha (t , \, \cdot)}\mathcal{H}$.

The Mabuchi metric in $\mathcal{H}$ is defined as follows;
for any two tangent vectors $\psi_1,\psi_2$ in the tangent space of $\mathcal{H}$ at $\phi$, it is given by
\begin{align}\label{eqn: definition of mabuchi metric}
 Ma_{\phi}(\psi_1,\psi_2)=\int_M\psi_1\psi_2\frac{\om_\phi^n}{n!} \; .
\end{align}
In the given K\"ahler class $\Om$, the corresponding volume forms form a conformal class.
This is the second description of the space of K\"ahler metrics.
\begin{defn}\label{conformal}
The \emph{space of volume conformal factors} is given by the following space
\begin{align*}
 \mathcal{C}:=
 \left\{
   F\in C^{\infty}(M , \, \mathbb{R} )\, \left|  \,
   F> 0 , \, \int_M F \omega^n = \int_M \omega^n  \right.
 \right\}\;.
\end{align*}
\end{defn}
The notions of differentiable curves,
tangent vectors, curves of tangent vectors on $\mathcal{C}$ are in exactly the same vein as
it was for the space $\mathcal{H}$.
For any fixed conformal factor $F\in \mathcal{C}$, we have that the tangent space of
$\mathcal{C}$ at $F$ is
\begin{align}\label{eqn:tangent vectors}
 T_F\mathcal{C}=
 \left\{
 G\in C^{\infty}(M,\mathbb{R})
 \; \left|  \;
 \int_M G \omega^n = 0
 \right.
 \right\}\; .
\end{align}

The Monge-Amp\`ere map, defined by
\begin{align*}
MA : \mathcal{H} &\longrightarrow \mathcal{C}\\
\phi &\mapsto MA(\phi) := F(\phi) = \left( \frac{ \omega_\phi^n}{\omega^n} \right) ,
\end{align*}
is, by the positive answer to \emph{the Calabi volume conjecture} \cite{Yau},
a diffeomorphism between $\mathcal{H}$ and $\mathcal{C}$.
Its differential at $\phi \in \mathcal{H}$ is
\begin{align*}
MA_* [\phi] : T_{\phi} \mathcal{H} &\longrightarrow T_F \mathcal{C}\\
\psi & \mapsto G(\psi): = \left( \frac{ \omega_\phi^n}{\omega^n} \right) \Delta_{\phi}\psi \; ,
\end{align*}
where the expression of $\Delta_\phi$ in a local coordinate system $(U ,\, z^1 , \, \cdots , \, z^n)$
is $\Delta_\phi \psi = g_\phi^{j\overline{k}}\psi_{j\overline{k}}$, and $g_\phi$ is the K\"ahler
metric corresponding to the K\"ahler form $\omega_\phi$.

The Calabi metric is defined for any $\psi_1 , \psi_2 \in T_\phi \mathcal{H}$,
\begin{align}\label{calabimetric}
Ca_{\phi} (\psi_1 , \psi_2 ) : = \int_M\Delta_\phi \psi_1 \Delta_\phi \psi_2\frac{\omega_\phi^n}{n!} \; .
\end{align}
To be precise, Mabuchi computed that, if $\phi = \phi (s , t)$ is
a smooth two parameter family of curves in the space of K\"ahler metrics $\mathcal{H}$,
and the corresponding two parameter families of curves of tangent vectors $\phi_t$, $\phi_s$ along $\phi$
are $\mathbb{R}$-linearly independent,
then the sectional curvature of the plane spanned by $\phi_s$ and $\phi_t$
is expressed in terms of their Poisson bracket as
\[
 \{ \phi_s , \phi_t \}_\phi =
\frac{\sqrt{-1}}{2}
\left(
g^{i\overline{j}} \frac{\partial \phi_s}{ \partial z^i }\frac{\partial \phi_t}{ \partial z^{\overline{j}}}
-
g^{i\overline{j}} \frac{\partial \phi_t}{ \partial z^i }\frac{\partial \phi_s}{ \partial z^{\overline{j}}}
\right)
=
Im(\partial \phi_s , \overline{\partial} \phi_t )_\phi \; .
\]
The expression of the sectional curvature $K_M$ for the Mabuchi metric is therefore
\[
 K_M (\phi_s , \phi_t )_\phi = -\frac{\int_M Im (\partial \phi_s , \overline{\partial} \phi_t)_\phi^2
\frac{\omega_\phi^n}{n!}}
{\sqrt{\int_M \phi_s^2 \frac{\omega_\phi^n}{n!}}
\sqrt{\int_M \phi_s^2 \frac{\omega_\phi^n}{n!}}
- \int_M \phi_s \phi_t \frac{\omega_\phi^n}{n!}} \; .
\]
We read off the above equation that $K_M \leq 0$ for all the linearly independent
sections $\phi_s , \phi_t$. On the other side, the first author proved that, for any linearly
independent sections $\phi_s , \phi_t$  the sectional curvature for the Calabi metric $K_C$ is
\[
 K_C (\phi_s , \phi_t) = \frac{1}{4V}\; ;
\]
here $V$ is the volume of the manifold $M$.
Eugenio Calabi conjectured that the sectional curvature for the Dirichlet metric
is bounded above and below respectively by the Mabuchi and the Calabi sectional curvatures.
In \cite{Calamai},
the first author proved that if $M$ has complex dimension one,
then for any linearly independent sections $\phi_s , \phi_t$,
the sectional curvature for the Dirichlet metric vanishes;
\[
 K_D (\phi_s , \phi_t) =0 \; .
\]
As a corollary, for closed Riemann surfaces we have that
$$
K_{\mbox{Mabuchi}} \leq K_{\mbox{Dirichlet}} < K_{\mbox{Calabi}}.
$$
One aim of the present paper is to take a step further towards the above bounds then
the complex dimension of the manifold $M$ is bigger than one. (cf. Subsection \ref{subsection sectional curvature}.)
\smallskip
\section{The Dirichlet metric}
We now want to introduce the Dirichlet metric. We shall use the point of view of differential forms.
In the present section we are going to make repeatedly use of the Einstein notation.

\begin{defn}
We define the following space
\begin{align}\label{one form}
 \mathcal{A}:= \{ d\phi \, | \, \omega_\phi \mbox{ is a K\"ahler metric in } \mathcal{H} \} \; .
\end{align}
We note that the space $\mathcal{A}$ is a $\mathbb{R}$-convex, open subset in the space of real $1$-forms;
indeed, this follows from the fact that the space $\mathcal{H}$ is $\mathbb{R}$-convex, that is if $\phi_0 ,\, \phi_1$
are in $\mathcal{H}$, then for any $t\in[0,1]$ also $(1-t)\phi_0 + t\phi_1$ is in $\mathcal{H}$.
\end{defn}

\begin{lem}\label{lemma:H one to one with A}
The space $\mathcal{A}$ is in one-to-one correspondence
with the space of K\"ahler metrics $\mathcal{H}$.
\end{lem}
\begin{proof}
Consider the map $\Phi : \mathcal{H} \rightarrow \mathcal{A}$, which maps a K\"ahler metric $\omega_\phi$
to the element $d\phi \in \mathcal{A}$.
The map is well defined. Indeed, if two K\"ahler metrics $\omega_{\phi_1}$
and $\omega_{\phi_2}$ coincide,
then they differ by an additive real constant, say $\phi_1 = \phi_2 +C$.
Thereby, we have $d\phi_1 = d\phi_2$.
The map is surjective by definition of the space $\mathcal{A}$.
Concerning the injectivity, if we have $d\phi_1 = d\phi_2$,
then we also have $\partial \overline{\partial}\phi_1 = \partial \overline{\partial}\phi_2$
and thereby the corresponding K\"ahler metrics
$\omega_{\phi_1}$ and $\omega_{\phi_2}$ coincide.
\end{proof}

\smallskip
As an application of Lemma \ref{lemma:H one to one with A}, we can define differentiable curves on $\mathcal{A}$ as
the image under the map $\Phi$ of the already defined differentiable cruves on $\mathcal{H}$. And using differentiable curves,
or just differentiating the map $\Phi$, we conclude that the tangent space of $\mathcal{A}$  at a point $d\phi$ is given by
\[
T_{d\phi}\mathcal{A} =\left\{ d\psi \, | \,  \psi \in C^{\infty} (M , \mathbb{R}) \right\} \, ,
\]
that is, it is the space of exact real $1$-forms on $M$.

\begin{rmk}\label{rmk: almost complex structure}
The infinite dimensional Riemannian manifold $\mathcal{A}$ carries a natural almost complex structure $J$ defined as
\begin{align*}
J\p\psi=i\p\psi; J\bar\p\psi=-i\bar\p\psi.
\end{align*}
\end{rmk}

We are ready to present a central notion of this paper.
\begin{defn}\label{defn:dirichlet metric}
Consider a point $d\phi \in \mathcal{A}$ and two tangent vectors $d\psi , d\chi \in T_{d\phi}\mathcal{A}$;
let $g_\phi$ be the K\"ahler metrics corresponding to the element $d\phi \in \mathcal{A}$.
Consider a local coordinate chart $(U , \, z_1 , \cdots , \, z_n)$ on $M$.
Then, let the pairing $(d\psi , \, d\chi)_{g_{\phi}}$ be defined, at a point $p\in U$, by
$$
(d\psi , d\chi)_{g_{\phi}} = g_{\phi}^{i\overline{j}}(\psi_i \chi_{\overline{j}} + \psi_{\overline{j}} \chi_{i}) \, ,
$$
where the Einstein notation is used.
Now, the Dirichlet metric on $T_{d\phi}\mathcal{A}$ can be given as
\[
Di_{d\phi}<d\psi , d\chi >:=
<d\psi , d\chi >_{g_\phi} =
\int_M (d\psi , d\chi)_{g_{\phi}}  \frac{\omega_\phi^n}{n!} \; .
\]
\end{defn}
As our goal is to prove the features of the geometry arising from the Dirichlet metric,
it is natural to start with the study of the Levi-Civita covariant derivative of that metric.
\medskip
\subsection{Levi-Civita covariant derivative}
The Levi-Civita covariant derivative on the space $\mathcal{H}$ endowed with a Riemannian structure
is, analogously to the case of standard Riemannian geometry, a way of differentiating tangent vector fields along
curves on $\mathcal{H}$. For the precise definition of the Levi-Civita covariant derivative,
we refer to \cite[Definition 7]{Calamai}.
As remarked in \cite{Don} its existence needs to be proved  in our infinite dimensional environment,
unlike the finite dimensional Riemannian geometry theory.
The proof that the Levi-Civita covariant derivative for the Dirichlet metric exists
can be found in \cite{Calamai}.
Here, for the reader's convenience,
we are going to represent those results with the slightly different approach of the previous section.
To start off, we introduce a notation which will be very useful here and in the remainder of the paper.
\begin{defn}\label{tensorC}
Fix an element $d\phi \in \mathcal{A}$ and its corresponding K\"ahler form $\omega_\phi$.
Let $f\in C^{\infty}(M ,\, \mathbb{R})$ be a smooth real valued function.
We define the $(1,1)$ differential form $C[f]$
by $C[f] : = (\Delta_\phi f) \omega_\phi - i \p \bar\p f $, where we recall that in a coordinate
chart $(U , \, z^1 , \cdots ,\, z^n )$, the expressions of the K\"ahler form $\omega_\phi$ and
of its corresponding K\"ahler metric $g_\phi$ are $g_{\phi} = g_{a\bar b}dz^a \otimes dz^{\bar b}$
and $\omega_\phi = g_{a\bar b} dz^a \wedge dz^{\bar b}$.
\end{defn}
\smallskip
\begin{lem} \label{lemmaclosedness}
Fix an $f\in C^{\infty}(M , \mathbb{R})$; after raising the indices of the differential form
$C[f]$, and and fixing any index $i$, we have, in a normal coordinate chart,
\[
 \sum_{j=1}^n (C[f]^{i\overline{j}})_{,\overline{j}} = 0.
\]
\end{lem}
\begin{proof}
 Fix a point $p$ in a normal coordinate chart $(U , z_1 , \cdots , z_n )$. For simplicity, we write $g$ instead of $g_\phi$.
 The operation of raising the indexes, when applied to $C[f]$ gives, for the component $i, \bar j$, the following output,
 (where for simplicity we write $\Delta$ instead of $\Delta_\phi$)
 \[
  C[f]^{i\bar j} =
  \left[
  (\Delta f) g^{i\bar b} g_{a\bar b} g^{a \bar j} - g^{i\bar b} \frac{\p^2 f}{\p z^a \p z^{\bar b}} g^{a \bar j}
  \right]
  \; .
 \]
Now, when we differentiate it with respect to $\frac{\p}{\p z^{\bar j}}$, we repeatedly use that, for any
indexes $a, b$, then in normal coordinates there holds $\frac{\p}{\p z^{\bar j}} g^{a\bar b} = 0$; we get that
\begin{align}\label{eqn: first addendum of C[f]}
 \frac{\p}{\p z^{\bar j}} \left[(\Delta f) g^{i\bar b} g_{a\bar b} g^{a \bar j} \right]
 = \left[\frac{\p}{\p z^{\bar j}} (\Delta f) g^{i\bar b} \right] g_{a\bar b} g^{a \bar j}
 = \left(\Delta \frac{\p}{\p z^{\bar j}} f \right) g^{i\bar b} g_{a\bar b} g^{a \bar j} \, ,
\end{align}
while
\begin{align}\label{eqn: second addendum of C[f]}
 \frac{\p}{\p z^{\bar j}}\left[ g^{i\bar b} \frac{\p^2 f}{\p z^a \p z^{\bar b}} g^{a \bar j} \right]
 =  g^{i\bar b} \frac{\p^3 f}{\p z^a \p z^{\bar b}\p z^{\bar j}} g^{a \bar j}.
\end{align}
Now, performing a summation with respect to the indexes $a, b , j$, we find that the two addenda \eqref{eqn: first addendum of C[f]}
and \eqref{eqn: second addendum of C[f]} of $C[f]$ coincide.
This completes the proof of the lemma.
\end{proof}
\smallskip

Another formula which will be useful in the remainder is the derivative of the Laplacian operator.

\begin{defn}\label{defn:asterisque}
Let us consider a pair of $(1,1)$-differential forms  $S, T $ and a
one-covariant tensor $\alpha$. Suppose that in a same coordinate chart $(U , \, z^1 , \, \cdots , \, z^n)$ there
holds $T = T_{j\bar k} dz^j \wedge dz^{\bar k}$, $S = S_{j\bar k} dz^j \wedge dz^{\bar k}$
and $\alpha = \alpha_{j} dz^j + \alpha_{\bar j} dz^{\bar j}$. We define the function $S*T$ at a point $p\in U$ as
$$
S * T :=
g_\phi^{j\overline{k}}g_\phi^{p\overline{q}}S_{j\bar q}
T_{p\bar k }.
$$
Notice that the operation $S*T$ is symmetric in $S$ and $T$.

We also need to label as well by $*$ the following contraction of a twice-covariant tensor $T$ and a
one-covariant tensor $\alpha$. The one-covariant
tensor $T*\alpha$ is defined, at a point $p\in U$, as
\begin{align}\label{contraction two tensor and one tensor}
 T*\alpha = g_\phi^{a\bar b} T_{a \bar k} \alpha_{\bar b} dz^{\bar k}
 +g_\phi^{a\bar b} T_{ k\bar b} \alpha_{a} dz^{ k} .
\end{align}
\end{defn}

\begin{rmk}
Consider a two parameter family of curves on $\mathcal{H}$; namely,
$\phi = \phi (s,t) : [0,1]\times[0,1] \rightarrow \mathcal{H}$. Thus we have the formula
\begin{align}\label{star}
\frac{\partial}{\partial t} \left( \Delta_\phi \frac{\partial \phi}{\partial s} \right) =
\Delta_\phi \frac{\partial^2 \phi}{\partial s \partial t}
-  \p \bar \p\frac{\partial \phi }{\partial t} *  \p \bar \p \frac{\partial \phi }{\partial s}.
\end{align}
\end{rmk}

\begin{defn}
Fix a smooth curve $\phi = \phi (t) : [0,1] \rightarrow \mathcal{H}$; also, let
$\psi = \psi (t)$ be a smooth curve of tangent vectors along the curve $\phi$,
with $\psi (t) \in T_{\phi(t)}\mathcal{H}$ for any $t\in [0,1]$.
Write $\phi_t$ for $\frac{\p \phi}{\p t}$.
Then, let $f[\phi_t , \psi] \in C^{\infty}(M,\mathbb{R})$ be such that
$\Delta_\phi f[\phi_t , \psi] :=  \p \bar \p \psi * \p \bar \p \phi_t - \Delta_\phi \psi \cdot \Delta_\phi \phi_t$.
Notice that, by the above remark, the function $f[\phi_t , \psi]$ is well defined.
\end{defn}

As we are dealing with real differential forms,
by means of the Hodge theory we decompose any real $1-$form $\alpha$ as
\[
\alpha = da + d^* \beta + h_\alpha ,
\]
where $a\in C^{\infty}(M , \mathbb{R})$, $\beta$ is a real $2-$form and $h_\alpha$
is the harmonic part of $\alpha$; that is, we have $dh_\alpha =0$ and $d^* h_\alpha =0$.
We also shall use the following notation for the natural projections;
\begin{align}\label{defnprojections}
 \pi_d (\alpha) =da, \qquad \pi_{d^*} (\alpha ) = d^* \beta ,
\qquad \pi_{harm} (\alpha) = h_{\alpha}\; .
\end{align}
  The next result is preliminary for the main one of this section.
\begin{lem}
 Let $d\phi(t)$ be a smooth curve in the space $\mathcal{A}$ and let  $d\psi = d\psi (t)$
 be a smooth curve of tangent vectors along the curve $d\phi$,
with $d\psi (t) \in T_{d\phi(t)}\mathcal{A}$ for any $t\in [0,1]$.
Then we have
\begin{align}\label{meanzero}
 \int_M ( \p \bar \p \psi * \p \bar \p \phi_t - \Delta_\phi \psi \cdot \Delta_\phi \phi_t) \frac{\omega_\phi^n}{n!} =0\; ;
\end{align}
also we have, in the notation of Definition \eqref{contraction two tensor and one tensor},
\begin{align}\label{theoryandpractice}
 df[\phi_t , \psi] = \pi_d
\left(
\frac{1}{2} C[\phi_{t} ]* d\psi
\right)\; .
\end{align}
\end{lem}
\smallskip
\begin{proof}
Concerning \eqref{meanzero}, it is readily computed as
\[
 0= - \frac{\partial}{\partial t} \int_M \Delta_\phi \psi  \frac{\omega_\phi^n}{n!} =
\int_M ( \p \bar \p \psi * \p \bar \p \phi_t - \Delta_\phi \psi \cdot \Delta_\phi \phi_t) \frac{\omega_\phi^n}{n!}\; .
\]

About \eqref{theoryandpractice} let us call for convenience
\[
 \alpha =
\frac{1}{2} C[\phi_{t} ]* d\psi\; ;
\]
then, let $\tilde{f}$ be any real valued smooth function on $M$ and let us consider the integral
\begin{align*}
& \int_M \tilde{f} \cdot d^*\alpha \frac{\omega_\phi^n}{n!}
= \int_M (\alpha , d\tilde{f})_{g_\phi}  \frac{\omega_\phi^n}{n!}\\
\end{align*}
which we integrate by parts again to get
\begin{align*}
&= \int_M \tilde{f}\cdot\left( \p \bar \p \phi_t * \p \bar \p \psi - \Delta_\phi \phi_t \cdot \Delta_\phi \psi \right)
      \frac{\omega_\phi^n}{n!}\, ,\\
\end{align*}
where at the last equality we applied Lemma \ref{lemmaclosedness}.
From this we infer that
\[
 d^*\alpha  =  \p \bar \p \phi_t * \p \bar \p \psi - \Delta_\phi \phi_t \cdot \Delta_\phi \psi;
\]
but, in view of \eqref{meanzero}
we deduce that $ \p \bar \p \phi_t * \p \bar \p \psi - \Delta_\phi \phi_t \cdot \Delta_\phi \psi = \Delta_\phi f$,
for some smooth real valued function $f$.
Putting our latest equalities together, we get $d^* \alpha = d^* d f$,
that is $\pi_d (\alpha ) = df$ and \eqref{theoryandpractice} is proved.
This concludes the proof of the lemma.
\end{proof}
\smallskip
We are now ready to present the covariant derivative for the Dirichlet metric in this environment.
\smallskip
\begin{prop}
Let $d\phi(t)$ be a smooth curve in the space $\mathcal{A}$ and
let $d\psi = d\psi (t)$ be a smooth curve of tangent vectors along the curve $d\phi$,
with $d\psi (t) \in T_{\phi(t)}\mathcal{A}$ for any $t\in [0,1]$;
 then the formula
\begin{align}\label{Levi-civita_con_der_second_form}
D_t d\psi = d\frac{\partial \psi}{\partial t}
+ \frac{1}{2} \pi_d \left(
C[\phi_{t} ]*d\psi
\right)
\end{align}
gives the Levi-Civita covariant derivative for the Dirichlet metric on the space $\mathcal{A}$.
\end{prop}
\smallskip
\begin{proof}
First of all, let us notice that the right hand side of \eqref{Levi-civita_con_der_second_form}
is an exact one form; that is the operator $D_t$ takes values in $T_{d\phi}\mathcal{A}$,
as it must be.

About the metric compatibility (for its very definition, see \cite[Definition 7, (1c)]{Calamai}), we compute
\begin{align}\label{metric_compatibility}
& \frac{\partial}{\partial t} < d\psi , d \chi >_{d\phi} =
\frac{\partial}{\partial t} \int_M (d\psi , \, d\chi)_{g_{\phi}}
 \frac{\omega_\phi^n}{n!} \nonumber\\
& = - \int_M [(\p \bar \p \phi_t * \p \psi , \, \bar \p \chi )_{g_\phi}
+ (\p \bar \p \phi_t * \p \chi , \, \bar \p \psi )_{g_\phi} ] \frac{\omega_\phi^n}{n!} \nonumber\\
& + \int_M [(d\psi_t , \, \chi )_{g_{\phi}}
+  (d\chi_t , \, \psi )_{g_{\phi}} ]
  \frac{\omega_\phi^n}{n!} \nonumber\\
&  +\int_M (d\psi , \, d\chi )_{g_{\phi}} \left( \Delta_{\phi}
\frac{\partial \phi }{\partial t} \right) \frac{\omega_\phi^n}{n!}. \nonumber\\
\end{align}
Inserting the equation \eqref{Levi-civita_con_der_second_form} into the equation
\eqref{metric_compatibility} we have, by means of \eqref{theoryandpractice}, the compatibility, i.e.
\[
\frac{\partial}{\partial t} < d\psi , d \chi >_{g_\phi} =
< D_t d\psi , d\chi  >_{g_\phi}  + < d\psi , D_t d\chi  >_{g_\phi}.
\]

In order to show that the expression \eqref{Levi-civita_con_der_second_form} is torsion free
(see \cite[Definition 7, (1d)]{Calamai}),
let  $d\phi (s, t)$ be a smooth two-parameter family of curves on $\mathcal{A}$,
and let us compute
\begin{align}\label{torsionfreetwo}
& D_t d\phi_s - D_s d\phi_t= \nonumber \\
&= \frac{1}{2}  \pi_d \left(
C[\phi_{t}] *d \phi_s
-C[\phi_{s}]* d\phi_t
\right) .
\nonumber \\
\end{align}
Now we claim that, for any $d\tilde{f} \in T_{d\phi} \mathcal{A}$, we have
\begin{align}\label{torsionfreeclaim}
<D_t d\phi_s - D_s d\phi_t , d\tilde{f} >_{g_\phi} = 0.
\end{align}
Indeed we have
\begin{align*}
&2<D_t d\phi_s - D_s d\phi_t , d\tilde{f} >_{g_\phi}  \\
&= \int_M \left( \frac{1}{2}   \left[
C[\phi_{t}] * d\phi_s -C[\phi_{s}] * d\phi_t
\right]  , d\tilde{f} \right)_{g_\phi}  \frac{\omega_\phi^n}{n!}\\
&= -\int_M   \left\{
\left(
\p \bar \p \phi_{t}* \p \bar \p \phi_s
 +  \Delta_\phi \phi_{t} \Delta_\phi \phi_s
\right) \tilde{f}
-
\left(
\p \bar \p \phi_{t}* \p \bar \p \phi_s
 +  \Delta_\phi \phi_{t} \Delta_\phi \phi_s
\right)
\tilde{f} \right\}\frac{\omega_\phi^n}{n!}=0,\\
\end{align*}
where at the first equality we  used the fact that the projections \eqref{defnprojections} are
orthogonal with respect to the Dirichlet metric;
while at the last equality we integrated by parts and we applied Lemma \ref{lemmaclosedness}.
Thereby, claim \eqref{torsionfreeclaim} is proved and the expression
\eqref{Levi-civita_con_der_second_form} is the Levi-Civita covariant derivative for
the space $\mathcal{A}$ endowed with  the Dirichlet metric.
\end{proof}
\medskip
In the following discussion we introduce the notation $D_t \phi_s$ which closely resembles the
already defined element $D_t d\phi_s$. In fact, letting $\phi_s $ and $d\phi_s$ be corresponding elements
via the map $\Phi_*$ defined in Lemma \ref{lemma:H one to one with A},
it turns out that $D_t \phi_s$ which we are going to define corresponds to the already defined $D_t d\phi_s$.
This will be useful in Remark \ref{remark: one}.
\begin{rmk}
We notice that, with the notation of Definition \ref{tensorC}, we have the following nice formula
\begin{align*}
\frac{\partial}{\partial t} < d\psi , d\chi >_{g_\phi} =
\int_M \left( ( C[ \phi_t ] * d\psi , d\chi )_{g_\phi}
+ ( d\frac{\partial \psi}{\partial t}, d \chi )_{g_\phi}
+ ( d\psi , d\frac{\partial \chi}{\partial t})_{g_\phi}\right) \frac{\omega_\phi^n}{n!}.
\end{align*}
\end{rmk}
\medskip
\begin{rmk}
In the notation of the space $\mathcal{H}$,
the Dirichlet metric is the pairing
\[
 < \chi , \psi >_\phi = \frac{1}{2}\int_M (d\chi , d\psi )g_\phi \frac{\omega_\phi^n}{n!},
\]
where $\psi, \chi \in T_\phi \mathcal{H}$.
By means of the uniqueness of the Levi-Civita covariant derivative, which still holds in this infinite
dimensional environment (see \cite{Don}),
our expression \eqref{Levi-civita_con_der_second_form}
is equivalent to what we had found in \cite{Calamai},
that is
\begin{align}\label{levicivita}
\Delta_\phi D_t \psi = \Delta_\phi \frac{\partial \psi}{\partial t}  +
\left( \Delta_\phi  \frac{\partial \phi}{\partial t} \right)
( \Delta_\phi \psi ) + \frac{\partial ( \Delta_\phi \psi ) }{\partial t}.
\end{align}
\end{rmk}
\bigskip
\subsection{Sectional curvature}\label{subsection sectional curvature}
In the present subsection we are going to prove that the sectional
curvature of the Dirichlet metric for a two plane
spanned by tangent vectors $\phi_s , \phi_t \in T_\phi \mathcal{H}$
is bounded above and below by a constant $K$ which depends only on the point
$\phi$ and on one generator $\phi_s$.
This led us to conjecture that unlike the case when $M$ is a Riemann surface,
the sectional curvature is not flat.

The main result of the section will be the following.
\begin{thm}\label{thmsectionalcurvature}
 If $d\phi(s,t)$ a smooth two parameter family of curves in $\mathcal{A}$
such that $d\phi_s , d\phi_t$ are $\mathbb{R}$-linearly independent at $d\phi$,
than there exists a positive constant $K=K(\omega_\phi , d\phi_s)$,
such that we have the following bounds for the sectional curvature of the Dirichlet metric
$$
 -K(\omega_\phi , d\phi_s)
\leq
K_{\mbox{ \emph{Dirichlet} }}   (d\phi_s, d\phi_t)_\phi
\leq
K(\omega_\phi , d\phi_s) .
$$
\end{thm}
\smallskip
In what follows we simplify further the notation,
writing $\Delta $ for $\Delta_\phi$;
\begin{defn}
 Given  $d\phi = d\phi (s, t)$ a smooth two-parameter family in $\mathcal{A}$,
we call $A= A(s,t)$ and $B= B(s,t)$ the following real valued smooth functions on $M$.
\begin{align*}
 A(s , t) &: = D_t \phi_s ;\\
B(s , t) &:= \Delta A(s ,t) = 2 \Delta \phi_{st}
+ \Delta \phi_t \Delta \phi_s - \p \bar \p \phi_t * \p \bar \p \phi_s .
\end{align*}

Remember that $d\phi_s \in T_{d\phi}\mathcal{A}$ and
$\phi_s \in T_{\phi} \mathcal{H}$ corresponds via the map $\Phi_*$ of Lemma \ref{lemma:H one to one with A}.
Notice that $A$ and $B$ are symmetric with respect to the real parameters $s$ and $t$.
\end{defn}
\smallskip
\begin{lem}\label{lem1}
Let $d\phi = d\phi (s, t , \s , \tau )$ be a four parameter family of curves
in the space of K\"ahler metrics $\mathcal{A}$.
Then, with the above notation we have
\begin{align}\label{equaone}
\int_M D_\sigma D_\tau \phi_s \Delta \phi_t \frac{\omega_\phi^n }{n!} =\int_M \phi_t
\left[
2 B_\sigma (s , \tau)
+ \Delta \phi_{\sigma} B(s , \tau)
+ \p \bar \p \phi_{\sigma} * \p \bar \p  A(s, \tau)
\right]
\frac{\omega_\phi^n }{n!}. \nonumber \\
\end{align}
\end{lem}
\begin{proof}
We compute, with an integration by parts
\begin{align} \label{first}
 &\int_M   D_\s D_\t \phi_s  \Delta \phi_t  \frac{\omega_\phi^n}{n!} \nonumber \\
&=
 \int_M \phi_t \cdot \Delta (D_\s D_\t \phi_s )  \frac{\omega_\phi^n}{n!}
\nonumber\\
&= \int_M \phi_t \cdot
\left[
2\frac{\partial}{\partial \s} \Delta D_\tau \phi_s + \Delta \phi_\s \Delta D_\tau \phi_s
+ \p \bar \p \phi_\s * \p \bar \p  D_\tau \phi_s
\right]
\frac{\omega_\phi^n}{n!};  \nonumber \\
\end{align}
the latter expression of \eqref{first} is precisely the right hand side of the claimed
\eqref{equaone}, and thus the lemma is proved.
\end{proof}
\medskip
Before moving further, we fix another piece of notation.
\begin{defn}
 Let $\phi = \phi (s, t): [0,1] \times [0,1] \rightarrow \mathcal{H}$,
a smooth two-parameter family of curves  in $\mathcal{H}$.
Define, for any fixed $s,t \in [0,1]$, $a(s,t) \in C^{\infty}(M ,\, \mathbb{R})$ to be
\begin{align}\label{notationtwo}
 a(s , t) := A(s , t) - 2 \phi_{st} \;  .
\end{align}
\end{defn}
\smallskip
\begin{rmk}\label{remark: one}
For any fixed pair $s,t \in [0,1]$, then the function $a(s , t)$ is actually an
element of $T_{\phi(s,t)}\mathcal{H}$.
Indeed, what we have to check is
\[
 \int_M a(s,t) \frac{\omega_\phi^n}{n!} =0 \; .
\]
On the first addendum of $a(s,t)$, we have that
\[
 \int_M D_s \phi_t \frac{\omega_\phi^n}{n!} =0 \, ,
\]
by the very definition of covariant derivative. Also,
on the second term of $a(s,t)$ we have
\[
 \int_M -2\phi_{st} \frac{\omega_\phi^n}{n!} =
\frac{\partial}{\partial s}\int_M  -2 \phi_t \frac{\omega_\phi^n}{n!}
+2 \int_M  \phi_t \Delta \phi_s \frac{\omega_\phi^n}{n!} \; .
\]
The latter vanishes if and only if $\phi_s$ and $\phi_t$ are orthogonal
and in this case, $a(s,t)$ is an element of $T_{\phi(s,t)}\mathcal{H}$.
\end{rmk}
\begin{lem}
 With the notation as in \eqref{notationtwo},  for a two parameter family of curves
$d\phi = d\phi(s , t)$ in $\mathcal{A}$, we have  the formula
\begin{align}\label{claimedtwo}
 &\int_M \{ D_s D_t \phi_s - D_t D_s \phi_s \}  \Delta \phi_t  \frac{\omega_\phi^n}{n!}= \nonumber \\
&= \int_M \phi_t \{ \p \bar \p \phi_s * \p \bar \p  a( s, t) - \Delta \phi_s \Delta a( s , t )
                - \p \bar \p \phi_t * \p \bar \p  a (s , s) + \Delta \phi_t \Delta a(s , s) \} \frac{\omega_\phi^n}{n!} \; .
\nonumber \\
\end{align}
\end{lem}
\begin{proof}
We write more explicitly the formula \eqref{equaone}, substituting the definition of $B(\tau)$;
what we get is
\begin{align}\label{intermediateone}
 &\int_M  D_\sigma D_\tau \phi_s \Delta \phi_t \frac{\omega_\phi^n}{n!} = \nonumber \\
&= \int_M \phi_t \{ 2 B_\sigma (s , \tau) + \Delta \phi_\sigma \cdot B( s, \tau)
+ \p \bar \p \phi_\sigma * \p \bar \p A( s, \tau) \}   \frac{\omega_\phi^n}{n!} = \nonumber \\
&= \int_M \phi_t \left\{ 4 \Delta \phi_{s\tau \s} - 4  \p \bar \p \phi_\s * \p \bar \p  \phi_{s\t}
-2 \p \bar \p \phi_\s * \p \bar \p \phi_\t \cdot \Delta \phi_s + 2 \Delta \phi_{\t \s} \Delta \phi_s
\right. \nonumber \\
&-2 \p \bar \p \phi_\s * \p \bar \p \phi_s \cdot \Delta \phi_\t + 2 \Delta \phi_{s\s} \Delta \phi_\t
-2 \p \bar \p  \phi_{\s\t} * \p \bar \p  \phi_s + 2 \p \bar \p \phi_\t * \p \bar \p \phi_{s\s} \nonumber \\
&\left. +2 \Delta \phi_\s \Delta \phi_{s\t} + \Delta \phi_\s \Delta \phi_\t \Delta \phi_s
-\Delta \phi_\s \p \bar \p  \phi_\t * \p \bar \p \phi_s + \p \bar \p  \phi_\s * \p \bar \p A(s , \tau) \right\} \frac{\omega_\phi^n}{n!} \; .
\nonumber \\
\end{align}
Thus, using \eqref{intermediateone}, by the mutual cancellation of the terms
symmetric in $\s$ and $\t$,  we get
\begin{align}\label{intermediatetwo}
 &\int_M \{ D_\sigma D_\tau \phi_s - D_\t D_\s \phi_s\} \Delta \phi_t \frac{\omega_\phi^n}{n!}
= \nonumber \\
&\int_M \phi_t \left\{ \p \bar \p \phi_\s * \p \bar \p (A(s , \t) -2\phi_{s\t}) -
\Delta \phi_\t \p \bar \p \phi_\s * \p \bar \p \phi_s  \right.
\nonumber \\
&\left. -\p \bar \p \phi_\t * \p \bar \p (A(s , \s) -2 \phi_{s \s })+
\Delta \phi_\s \p \bar \p \phi_\t * \p \bar \p \phi_s \right\} \frac{\omega_\phi^n}{n!} \nonumber \\
&= \int_M \phi_t \left\{  \p \bar\p \phi_\s *\p \bar \p  a(s , \t)  -
\Delta \phi_\t \p \bar \p \phi_\s * \p \bar \p \phi_s \right. \nonumber \\
&\left. -\p \bar \p \phi_\t * \p \bar \p a(s, \s)
+ \Delta \phi_\s \p \bar \p \phi_\t * \p \bar \p \phi_s \right\} \frac{\omega_\phi^n}{n!} \nonumber \\
\end{align}
From \eqref{intermediatetwo} we immediately obtain
\begin{align}\label{intermediatethree}
 &\int_M \{ D_s D_t \phi_s - D_t D_s \phi_s\} \Delta \phi_t \frac{\omega_\phi^n}{n!}
= \nonumber \\
&\int_M \phi_t \{  \p \bar \p \phi_s * \p \bar \p  a(s, t)  -
\Delta \phi_t \p \bar \p \phi_s * \p \bar \p \phi_s - \p \bar \p \phi_t * \p \bar \p a(s , s)
+ \Delta \phi_s \p\bar \p \phi_t * \p \bar \p \phi_s \} \frac{\omega_\phi^n}{n!} \; . \nonumber \\
\end{align}
Now, it is worth noticing that
\begin{align}\label{intermediatefour}
 -\Delta \phi_s \p \bar \p \phi_t * \p \bar \p \phi_s + \Delta \phi_t \p \bar \p \phi_s * \p \bar \p \phi_s
= \Delta \phi_s \Delta a(s,t) - \Delta \phi_t \Delta a(s,s) \; ;
\end{align}
the latter claim is obvious by taking the difference of
\[
 \Delta \phi_s \Delta a(s,t) = (\Delta \phi_s )^2 \Delta \phi_t - \Delta \phi_s \p \bar \p \phi_t * \p \bar \p \phi_s
\]
and
\[
 \Delta \phi_t \Delta a(s,s) = \Delta \phi_t (\Delta \phi_s)^2 - \Delta \phi_t \p \bar \p \phi_s * \p \bar \p\phi_s \;  .
\]
Thus, combining \eqref{intermediatethree} and \eqref{intermediatefour} we get
\begin{align*}
 &\int_M \{ D_s D_t \phi_s - D_t D_s \phi_s\} \Delta \phi_t \frac{\omega_\phi^n}{n!}
= \nonumber \\
&= \int_M \phi_t \{  \p \bar \p \phi_s * \p \bar \p  a(s,t)  - \Delta \phi_s \Delta a(s,t)
                      - \p \bar\p \phi_t * \p \bar \p a(s,s)  +  \Delta \phi_t \Delta a(s,s) \}
\frac{\omega_\phi^n}{n!}\, , \nonumber \\
\end{align*}
which is precisely the claimed formula \eqref{claimedtwo}. This concludes the proof of the lemma.
\end{proof}
\begin{rmk}
In the case of a Riemannian surface, for any two functions $\chi$ and $\psi$, we have that
$$
\Delta  \psi \Delta \chi  = \p \bar \p  \psi * \p \bar \p  \chi .
$$
When we insert this fact in \eqref{claimedtwo} we find back
the  conclusion, already got in \cite{Calamai}, that the space of K\"ahler
potentials endowed with the Dirichlet metric is flat.
\end{rmk}
\smallskip
A consequence of formula \eqref{claimedtwo} is the following
nice expression for the numerator of the sectional curvature
of the Dirichlet metric.
\smallskip
\begin{lem}
The following formula holds for a smooth two parameter family of curves
$d\phi = d\phi (s , t)$ in the space of K\"ahler metrics $\mathcal{A}$;
\begin{align}\label{claimedthree}
 &\int_M \{ D_s D_t \phi_s - D_t D_s \phi_s\} \Delta_\phi \phi_t \frac{\omega_\phi^n}{n!}
= \nonumber \\
&= \frac{1}{2}\int_M |d a(s,t)|_{g_\phi}^2\frac{\omega_\phi^n}{n!} - \frac{1}{2}
\int_M (d a(s,s) , d a(t,t))_{g_\phi} \frac{\omega_\phi^n}{n!}\, , \nonumber \\
\end{align}
where we recall that the symmetric expression $a(\s , \t)$ satisfies
\begin{align}\label{a formula}
\Delta_\phi a(\s , \t) = \Delta_\phi \phi_\s \Delta_\phi \phi_\t - \p \bar \p \phi_\s * \p \bar \p \phi_\t .
\end{align}
\end{lem}
\begin{proof}
In the argument here, we are going to write $\Delta, g, \omega$ instead of
$\Delta_\phi, g_\phi, \omega_\phi$.
We first consider this part of the right hand side of \eqref{claimedtwo};
\begin{align*}
& \int_M \left(
\Delta a(s,s) \Delta \phi_t \cdot \phi_t -
	  \p \bar \p a(s,s)*  \p \bar \p \phi_t \cdot \phi_t
	  \right)
	  \frac{\omega^n}{n!} \, ,\nonumber
\end{align*}
which we integrate twice by parts using formula \eqref{contraction two tensor and one tensor},
getting
\begin{align*}
 =& \int_M \left(
 -\frac{1}{2} (d   a(s,s), d \Delta \phi_t )_g \cdot \phi_t
 -\frac{1}{2}(d  a(s,s), d  \phi_t )_g \cdot \Delta \phi_t \right. \\
 &\left. \phantom{++}+\frac{1}{2}(d  a(s,s),\, d \Delta \phi_t )_g  \cdot \phi_t
 +\frac{1}{2}( d  a(s,s), \,  \p \bar \p \phi_t * d \phi_t )_g
	 	  \right)
	  \frac{\omega^n}{n!} \nonumber \\
=& \int_M \left(
 \frac{1}{2} a(s,s)\cdot( d \phi_t ,\, d \Delta \phi_t )_g
 + a(s,s) (\Delta^2 \phi_t) \cdot \phi_t \right. \\
 &\phantom{++}+ \frac{1}{2} a(s,s)(\Delta \phi_t)^2
 +\frac{1}{2}a(s,s)( d \Delta\phi_t, d  \phi_t )_g \\
 &\phantom{++}-a(s,s)(\Delta^2 \phi_t) \cdot \phi_t
 -a(s,s) ( d  \phi_t , d \Delta \phi_t )_g  \\
 &\phantom{++}\left. -\frac{1}{2}a(s,s)( d  \Delta \phi_t ,  d \phi_t )_g
 -a(s,s) \p \bar \p \phi_t * \p \bar \p \phi_t
	 	  \right)
	  \frac{\omega^n}{n!} 	
\end{align*}
and so, after obvious cancellations, we get
\begin{align*}
  =\int_M (
 a(s,s) (\Delta \phi_t)^2
  -a(s,s)\p \bar \p \phi_t * \p \bar \p \phi_t
	 	  )
	  \frac{\omega^n}{n!} \nonumber 	
\end{align*}
and finally, by the very definition of $a(t,t)$, we conclude
\begin{align}\label{parta}
 &=\int_M a(s,s)\Delta a(t,t)\frac{\omega^n}{n!}
       =-\frac{1}{2}\int_M (d a(s,s) , d a(t,t) )_{g} \frac{\omega^n}{n!}. \nonumber \\
\end{align}

Concerning the other part of the right hand side of \eqref{claimedtwo}, we similarly compute
\begin{align*}
& \int_M (-\Delta a(s,t) \Delta \phi_s \cdot \phi_t
	  + \p \bar \p a(s,t)* \p \bar \p \phi_s \cdot \phi_t ) \frac{\omega^n}{n!} \nonumber \\
	=& \int_M \left( \frac{1}{2} ( d a(s,t),\, d \Delta \phi_s )_g \cdot \phi_t
+ \frac{1}{2}( d a(s,t) ,\, d \phi_t )_g \Delta \phi_s \right. \\
&\left. \phantom{++}- \frac{1}{2}( d a(s,t) ,\, d \Delta \phi_s )_g \cdot  \phi_t
 -\frac{1}{2}( d a(s,t) ,\,  \p \bar \p \phi_s * d \phi_t )_g
 \right) \frac{\omega^n}{n!} \nonumber \\
	  \end{align*}
	  and again by parts, we obtain
\begin{align*}
&= \int_M \left( - a(s,t) (\Delta^2 \phi_s) \cdot \phi_t
- \frac{1}{2} a(s,t)  ( d \Delta \phi_s , d  \phi_t )_g
\right. \nonumber \\
&\phantom{++} - a(s,t) \Delta \phi_t \Delta \phi_s -\frac{1}{2} a(s,t) (d \phi_t , d \Delta \phi_s )_g
 \nonumber \\
&\phantom{++} +a(s,t)  (\Delta^2 \phi_s) \cdot \phi_t
	   + \frac{1}{2}a(s,t)(d \Delta \phi_s , d   \phi_t )_\g \nonumber \\
&\phantom{++} \left. +\frac{1}{2}a(s,t) ( d \Delta \phi_s , d \phi_t )_g
 + a(s,t) \p \bar \p \phi_s * \p \bar \p \phi_t \right)
	      \frac{\omega^n}{n!} \nonumber \\
\end{align*}
	    and so we conclude
	      \begin{align}\label{partb}
&=\int_M (-a(s,t)(\Delta \phi_s)^2 +a(s,t) \p \bar \p \phi_{s}* \p \bar \p \phi_{t} )\frac{\omega_\phi^n}{n!}\nonumber \\
&=\int_M -a(s,t)\Delta a(s,t)\frac{\omega_\phi^n}{n!}
       = \frac{1}{2}\int_M (d a(s,t) , d a(s,t))_{g_\phi} \frac{\omega_\phi^n}{n!}. \nonumber \\
\end{align}
Combining \eqref{parta} and \eqref{partb} we get the claimed formula \eqref{claimedthree}.
\end{proof}
\medskip
We are  now going to get another equivalent description of the same quantity
already considered in the previous lemma.
\begin{lem}
We have for a two parameter family of curves $d\phi = d\phi (s ,t)$
in the space of K\"ahler metrics $\mathcal{A}$, the formula
(which makes use of the piece of notation introduced in Definition \ref{tensorC}
and in \eqref{contraction two tensor and one tensor})
\begin{align}\label{intermediatefive}
 &\int_M \{ D_s D_t \phi_s - D_t D_s \phi_s\} \Delta_\phi \phi_t \frac{\omega_\phi^n}{n!}
\nonumber \\
&=\int_M ( C[a( s, t)]* \p \phi_s , \bar \p \phi_t )_{g_\phi} \frac{\omega_\phi^n}{n!}
-\int_M ( C[a( s, t)]* \p \phi_t , \bar \p \phi_t )_{g_\phi} \frac{\omega_\phi^n}{n!}.
\nonumber \\
\end{align}
\end{lem}
\begin{proof}
Plug in the formula \eqref{claimedtwo} the definition of $C[a(\cdot , \cdot)]$,
and then integrate by parts. The equation \eqref{intermediatefive} will be gotten after
applying Lemma \ref{lemmaclosedness}.
\end{proof}
\bigskip
In what follows now, our aim is to further simplify the formula \eqref{intermediatefive}
without loss of generality.
\begin{lem}\label{lemma_pages18_2and18_3}
The following formula holds
\begin{align}\label{equa_first_addendum_of_sect_curv}
 \left|
\int_M (C[a( s, t)]* \p \phi_s , \bar \p \phi_t )_{g_{\phi}} \frac{\omega_\phi^n}{n!}
\right|
\leq
K_1 (\phi_s)  \int_M |\p \phi_t |_{g_{\phi}}^2 \frac{\omega_\phi^n}{n!},
\end{align}
where $K_1$ is a positive constant, depending on the tensor $C[\phi_s]$.
\end{lem}
\begin{proof}
 We notice that, using integration by parts and Lemma \ref{lemmaclosedness},
we have
\begin{align}
\int_M (C[a( s, t)] * \p \phi_s , \bar\p \phi_t )_{g_\phi} \frac{\omega_\phi^n}{n!}
=
-\int_M (C[a( s, t)] * \p \bar \p  \phi_s) \cdot  \phi_t \frac{\omega_\phi^n}{n!}.
\end{align}
Next, we remark that
\begin{align} \label{eq_intermediate_lem_first_add_sect_curv}
 C[a( s, t)] * \p \bar \p  \phi_s
=
 C[\phi_s] * \p \bar \p a( s, t) ;
\end{align}
indeed, by definition
\[
  C[a( s, t)] * \p \bar \p \phi_s
=
 \Delta a(s,t) \Delta \phi_s - \p \bar \p a(s,t) * \p \bar \p\phi_s,
\]
and by its symmetric structure we obtain \eqref{eq_intermediate_lem_first_add_sect_curv}.
Also, we have
\begin{align}\label{eqn_step_one_fisrt_add_sect_curv}
 \left|
\int_M (C[\phi_s]* \p\bar \p a( s, t)) \cdot \phi_t \frac{\omega_\phi^n}{n!}.
\right|
\leq K_3 ( C[\phi_s] ) \|\p a(s,t) \| \|\p \phi_t \| \; .
\end{align}
Then,
\begin{align*}
&\left| \int_M a(s,t) \Delta_\phi a(s,t) \frac{\omega_\phi^n}{n!} \right|
=
\left|
\int_M (C[\phi_s]* \p\phi_{t}, \bar \p a(s,t) )_{g_\phi}  \frac{\omega_\phi^n}{n!}
\right| \nonumber \\
&\leq
\epsilon  \int_M |\p a(s,t)|_{g_\phi}^2 \frac{\omega_\phi^n}{n!}
+ K_4 (\epsilon , C[\phi_s])  \int_M |\p \phi_t|_{g_{\phi}}^2
\frac{\omega_\phi^n}{n!} \; .
\end{align*}
Thus, choosing $\epsilon = \frac{1}{2}$ we have
\begin{align}\label{eqn_step_two_fisrt_add_sect_curv}
 \int_M |\p a(s,t)|_{g_\phi}^2 \frac{\omega_\phi^n}{n!}  \leq
K_5(\phi_s)  \int_M |\p \phi_t|_{g_{\phi}}^2 \frac{\omega_\phi^n}{n!} \; .
\end{align}
Thus, the combination of equations \eqref{eqn_step_one_fisrt_add_sect_curv}
and \eqref{eqn_step_two_fisrt_add_sect_curv} gives the claimed formula
\eqref{equa_first_addendum_of_sect_curv}, whence the proof is complete.
\end{proof}
\smallskip
\begin{lem}\label{lemma_second_add_sect_curv}
 The following formula, concerning the second term appearing on the right hand side
of \eqref{intermediatefive}, holds
\begin{align}\label{eqn_second_add_sect_curv}
 \left|
\int_M ( C[a(s ,s )] * \p \phi_t , \bar \p \phi_t )_{g_\phi} \frac{\omega_\phi^n}{n!}
\right|
\leq
K_2 ( \phi_s ) \int_M |\p \phi_t|_{g_{\phi}}^2 \frac{\omega_\phi^n}{n!} \; .
\end{align}
\end{lem}
\begin{proof}
We immediately estimate
\[
 \left|
\int_M (C[a(s ,s )]* \p \phi_t , \bar \p \phi_t )_{g_\phi} \frac{\omega_\phi^n}{n!}
\right|
\leq
K_6 (C[a(s,s]) \int_M |\p \phi_t|_{g_{\phi}}^2 \frac{\omega_\phi^n}{n!}.
\]
 Then, as $a(s,s)$ only depends on $\phi_s$, we get that the constant $K_6 (C[a(s,s])$
is equal to some $K_2 (\phi_s)$ and the proof is achieved.
\end{proof}
We are now in position to prove the main result of the section \par
\smallskip
\noindent {\bf Proof of Theorem \ref{thmsectionalcurvature}.}
As we want to compute the sectional curvature
of a two plane, we have the freedom to replace the two sections $d\phi_s$
and $d\phi_t$ by two other sections which span the same plane and which are orthonormal.
The latter is achieved simply by employing the Gram-Schmidt method.
We still call the two orthogonal  sections $d\phi_s $ and $d\phi_t$.
Thus, the sectional curvature of the plane spanned by $d\phi_s$ and $d\phi_t$ is
\[
 K_D (d\phi_s , d\phi_t) =
\int_M \{ D_s D_t \phi_s - D_t D_s \phi_s\} \Delta \phi_t \frac{\omega_\phi^n}{n!}
\]
Formula \eqref{intermediatefive}, combined with
formulas \eqref{equa_first_addendum_of_sect_curv} and \eqref{eqn_second_add_sect_curv}, gives
\begin{align}\label{intermediateseven}
 &\left|\int_M \{ D_s D_t \phi_s - D_t D_s \phi_s\} \Delta \phi_t \frac{\omega_\phi^n}{n!}
 \right| \nonumber \\
&\leq
\left(K_1 (\phi_s) + K_2 (\phi_s)\right)
\cdot
\int_M |\p \phi_t|_{g_{\phi}}^2 \frac{\omega_\phi^n}{n!} = K(\phi_s), \nonumber \\
\end{align}
where we used the orthonormality assumption. This completes the proof. $\qquad \square$ \par
\medskip
\begin{rmk}\label{cur1}
In view of formula \ref{claimedthree}, we notice that whenever at some point $\phi \in \mathcal{H}$
there are two $\mathbb{R}$-linearly independent tangent vectors $\phi_s , \phi_t$ such that
$a(s,s)$ or $a(t,t)$ is constant, then $K_D (\phi_s , \phi_t )_\phi $ is bigger or equal to zero.
On the opposite side, whenever $a(s,t)$ is constant, then $K_D (\phi_s , \phi_t )_\phi $ is less or equal to zero.
This fact suggests that the Dirichlet metric has curvature whose sign changes.
If this conjecture were confirmed, the consequence would be that the Dirichlet metric behaves
differently both from the Mabuchi metric, whose sectional curvature is non positive,
and from the Calabi metric, whose sectional curvature is positive.
\end{rmk}
\smallskip
\begin{rmk}\label{rmk: complex hessian equation}
We notice here that the formula \eqref{notationtwo} resembles but is slightly different
from the complex Hessian equation; a difference is
that $\phi_s$ and $\phi_t$ are not convex as in the environment of that equation.
As pointed out in the previous remark, solving the modified complex Hessian equations
$a(s,s)=0$ , $a(s,t) = 0$ or $a(t,t)=0$ (see \eqref{a formula})
would detect the sign of the sectional curvature of the Dirichlet metric.
This gives a motivation to approach the study of this new kind of partial differential equations.
\end{rmk}
\bigskip
\subsection{Geodesic equation}
Geodesics in the Mabuchi metric environment enter as main characters the conjectures of Donaldson's program
\cite{Don}; moreover, both for the Mabuchi and the Calabi metric, they are the key tools to show that
the space $\mathcal{H}$ is endowed with the structure of a metric space. In fact, in both cases it turns out that
they are length minimizing.
Let us begin with two definitions which are by now classical, also in this infinite dimensional
environment.
\begin{defn}
 Let $d\phi = d\phi (t)$ be a smooth path in $\mathcal{A}$, with $t\in [a,b]$.
We define the energy of $d\phi$ as
\[
 Engy (d\phi) = \int_a^b \left\{ \int_M (d\phi_t , d\phi_t )_{g_\phi} \frac{\omega_\phi^n}{n!} \right\} dt .
\]
Also, we define the length of the same arc $d\phi (t)$ as
\[
 Lgth (d\phi) = \int_a^b
\left\{
\int_M (d\phi_t , d\phi_t )_{g_\phi} \frac{\omega_\phi^n}{n!}
\right\}^{1/2} dt .
\]
\end{defn}
\medskip
\begin{rmk}
 As well as in the finite dimensional Riemannian geometry theory, there still holds that geodesic arcs
can be defined equivalently both as extremal points of the Energy functional and as solutions
of the equation $D_t d\phi_t =0$.
The argument goes precisely in the same vein, that is by  analyzing the first
variation formula (cf. \cite{Mab}).
\end{rmk}
\medskip
\begin{defn}
A bijective map $\lambda : \mathcal{A} \rightarrow \mathcal{A}$
is called an isometry of $\mathcal{A}$ if, for every smooth path $d\phi (t)$  in $\mathcal{A}$,
then its image $\lambda \circ d\phi (t)$ is still a smooth path, and moreover
\[
 Lgth (\lambda \circ d\phi) = Lgth (d\phi).
\]
\end{defn}
\medskip
\begin{prop}
 Let $h \in Aut^0 (M)$ be any holomorphic automorphism of $M$ in the connected
component of the identity. Then, the mapping $h^* : \mathcal{A} \rightarrow \mathcal{A}$,
such that $d\phi \mapsto h^* (d\phi)$ is an isometry of $\mathcal{A}$ endowed with the Dirichlet metric.
\end{prop}
\smallskip
\begin{proof}
 Consider a smooth arc $d\phi = d\phi (t)$ in $\mathcal{A}$, with $t\in [a,b]$.
Denote by $d\chi$ the element of $\mathcal{A}$ such that
$h^* \omega = \omega + i\partial \overline{\partial}\chi $.
Define $d\xi := d\chi + h^* d\phi $.
Then,
\[
 \omega + i\partial \overline{\partial} \xi
= \omega + i \partial (d\chi + h^*d\phi) = h^* (\omega + i\partial \overline{\partial}\phi) >0,
\]
which shows that $d\xi$ belongs to $\mathcal{A}$.
Also, when we consider the time derivative, there holds $d\xi_t = h^*d\phi_t $; whence
\begin{align*}
 Lgth (h^*d\phi)
& =
\int_a^b
\left\{
\int_M (d\xi_t , d\xi_t )_{g_\xi} \frac{\omega_\xi^n}{n!}
\right\}^{1/2} dt \\
&=
\int_a^b
\left\{
\int_M ( h^* d\phi_t , h^* d\phi_t )_{g_\phi}
h^*\left(
\frac{\omega_{\phi}^n}{n!}
\right)
\right\}^{1/2} dt
= Lgth (d\phi),
\end{align*}
which completes the proof of the proposition.
\end{proof}
\smallskip
\medskip
We are now going to prove that another result in \cite{Mab} still
holds in our setting, that is for the Dirichlet metric.
\smallskip
\begin{defn}
 Let $X$ be an holomorphic vector field on $M$; $X\in H^0 (M , \mathcal{O} (TM)) $.
We denote by $X_{\mathbb{R}}$ the real vector field defined by $X_\mathbb{R} := X + \overline{X}$.
Also, we denote by $L_{X_{\mathbb{R}}}$ the Lie derivative of $\omega$ with respect to $X_\mathbb{R}$.
\end{defn}
\smallskip
\begin{prop}
Let $X$ be an holomorphic vector field on $M$ such that $L_{X_{\mathbb{R}}}(\omega)=0$.
Let be given a smooth path $d\phi = d\phi(t)$ in $\mathcal{A}$ such that $L_{X_{\mathbb{R}}}(d\phi)=0$
for all $t$. Then
\[
 L_{X_{\mathbb{R}}}(D_t d\psi_t) = D_t (L_{X_{\mathbb{R}}} d\psi),
\]
for all the smooth sections $d\psi$ along the path $d\phi$.
\end{prop}
\smallskip
\begin{proof}
 Recall that formula \eqref{Levi-civita_con_der_second_form} gives the following
 expression of the Levi-Civita covariant derivative
\[
 D_t d\psi = d \psi_t + \frac{1}{2} \pi_d \left( C[\phi_t] * d\psi \right) \; .
\]
So, for $\frac{\partial}{\partial t}$ and $\pi_d$ commute with $D_t$, it remains to check that
\[
 L_{X_{\mathbb{R}}} C[\phi_t] *\p \psi =  C [\phi_t] * (  \partial (X_{\mathbb{R}} \psi)).
\]
We claim that $L_{X_{\mathbb{R}}} C[\phi_t]=0$; once achieved the claim,
 the above formula follows clearly. Thus, we compute, at a point $p$ of a local
 coordinate chart $(U , \, z^1 , \, \cdots , \, z^n )$,
\begin{align*}
 &L_{X_{\mathbb{R}}} C[\phi_t] =
- L_{X_{\mathbb{R}}} \partial\overline{\partial} \phi_t
+ L_{X_{\mathbb{R}}} (\Delta_\phi \phi_t) \omega\\
&= -\frac{\partial}{\partial t } \partial\overline{\partial} ({\mathbb{R}} \phi)
+ L_{X_{\mathbb{R}}} \left( \frac{\partial}{\partial t} (\Delta_\phi \phi_t
+ \phi_t^{i\overline{j}} \phi_{ti\overline{j}})\right)  \omega\\
&= \frac{\partial}{\partial t } (X_{\mathbb{R}}(\Delta \phi_t) +
 \frac{1}{2}L_{X_{\mathbb{R}}}(\phi_t^{i\overline{j}} \phi_{ti\overline{j}}) ) \omega
=0,
\end{align*}
where we used repeatedly the assumptions $X_{\mathbb{R}}\phi =0$ and $L_{X_{\mathbb{R}}} (\omega)=0$.
Note that the argument for  proving that
$L_{X_{\mathbb{R}}} C[\phi_t] *\p \psi =  C [\phi_t] * (  \partial (X_{\mathbb{R}} \psi))$
goes exactly the same way. Whence, the proof of the proposition is complete.
\end{proof}
\smallskip
Let us also recall other classical definitions, rephrased for the space $\mathcal{A}$,
in view of the next proposition.
\begin{defn}
 Let $d\phi = d\phi (t)$ be a smooth path in $\mathcal{A}$, with $t\in [0,1]$.
We denote by $S=S(\phi)$ the scalar curvature of the metric $\omega_\phi$.
It is a very well known fact that the scalar
$\underline{S}:= \frac{1}{Vol}\int_M S(\phi) \frac{\omega_\phi^n}{n!}$
is a constant that only depends on the K\"ahler class  $[\omega_\phi]$.
We define $f:= f(\phi) $ as the zero-mean function satisfying $\Delta_\phi f = S - \underline{S}$.
Finally, the $K$-energy ( or Mabuchi energy) is defined as (compare \cite[Definition (3.1)]{Mabuchi on K energy})
\[
 \nu (\phi ) := \frac{1}{2} \int_0^1 \int_M ( d\phi_t , df  )_{g_\phi} \frac{\omega_\phi^n}{n!} dt.
\]
\end{defn}
\smallskip
\begin{prop}\label{prop: convexity at a cscK metric}
 Let $d\phi = d\phi (t)$ be a geodesic for the Dirichlet metric, with $t\in [0,1]$.
Then, the K-energy is convex at a point $d\phi ( t_0 )$ which corresponds to a constant
scalar curvature K\"ahler metric.
\end{prop}
\begin{proof}
 By means of the metric compatibility of the Levi-Civita covariant derivative, we compute
\[
 \frac{\partial^2}{\partial t^2} \nu (\phi) = \frac{1}{2} \int_M
\left[
(D_t d\phi_t , df )_{g_\phi}
+ (d\phi_t , D_t df)_{g_\phi}
\right]
\frac{\omega_\phi^n}{n!} .
\]
The first addendum vanishes since, by assumption on $d\phi$, there holds $D_t d\phi_t =0$.
So we reduced the question to the study of
\begin{align*}
I := \int_M  (d\phi_t , D_t df)_{g_\phi} \frac{\omega_\phi^n}{n!}.
\end{align*}
We preliminarily compute, using the very definition of $D_t$,
\begin{align}\label{equ_term_covxity_k_energy}
 &I= \int_M (d\phi_t , d \frac{\partial f}{\partial t}
+ i_{Re X}C[\phi_t] )_{g_{\phi}}\frac{\omega_\phi^n}{n!} \nonumber \\
&= \int_M
\left[
-\phi_t \Delta_\phi \frac{\partial f}{\partial t}
+ \frac{1}{4} (C[\phi_t]* d\phi_t , df)_{g_\phi}
\right]
\frac{\omega_\phi^n}{n!} \nonumber \\
&=
\int_M
\left[
-\phi_t \Delta_\phi \frac{\partial f}{\partial t}
+ \frac{1}{2}(\p \bar \p \phi_t  * \p \bar \p  \phi_t - (\Delta_\phi \phi_t)^2 ) f
\right]
\frac{\omega_\phi^n}{n!}.
\end{align}
Our next goal is to analyze the addendum
$\int_M -\phi_t \Delta_\phi \frac{\partial f}{\partial t} \frac{\omega_\phi^n}{n!}$.
First, notice that, using the notation suggested in Definition \ref{defn:asterisque},
we can write the scalar curvature as
$S= \omega_\phi * \rho_{\phi}$, where $\rho_{\phi}$ is the Ricci form corresponding to $\omega_0$.
Time-differentiating formula $S= \Delta_\phi f + \underline{S}$, we get
$\Delta_\phi f_t = S_t + \p \bar \p \phi_t * \p \bar \p  f$.
Also, time-differentiating $S= \omega_\phi * \rho_{\phi}$,
we get $S_t = -\Delta_\phi^2 \phi_t - \p \bar \p \phi_t * \rho_{\phi}$.
Inserting the latter expression into the former one, we end up with
\[
 \Delta_\phi f_t = -\Delta_\phi^2 \phi_t - \p \bar \p \phi_t * (\rho_{\phi}- \p\bar \p  f).
\]
Thus, we can rewrite the term we are presently concerned about as
\begin{align} \label{equ_first_add_of_I}
& \int_M -\phi_t \Delta_\phi f_t \frac{\omega_\phi^n}{n!} \nonumber \\
&= \int_M -\phi_t \{ -\Delta_\phi^2 \phi_t
+ \phi_t *
(\p \bar\p f - \rho_{\phi}) \} \frac{\omega_\phi^n}{n!} \nonumber \\
&= \int_M
\{
(\Delta_\phi \phi_t )^2 - ((\rho_{\phi} - \p \bar \p  f)* \p \phi_t , \, \bar \p \phi_t )_{g_\phi}
\}
\frac{\omega_\phi^n}{n!},
\end{align}
where at the last equality we used the fact that the tensor  $\rho_{\phi} - \p \bar \p  f$
is divergence free; this can be seen, using the second Bianchi identity, in local coordinate chart
$(U, z^1 , \, \cdots , \, z^n)$; namely,
$(\rho_{i\overline{j}} - f_{i\overline{j}} )^i = S_{\overline{j}} - S_{\overline{j}} = 0$.

Now, our aim is to expand the second term of \eqref{equ_term_covxity_k_energy}. We compute
\begin{align} \label{equ_second_add_of_I}
 &\frac{1}{2}\int_M
(\p \bar \p \phi_t * \p \bar \p  \phi_{t} - (\Delta_\phi \phi_t )^2 )f
\frac{\omega_\phi^n}{n!}
=
\frac{1}{2}\int_M
[\p \bar \p \phi_t * \p \bar \p f - \Delta_\phi \phi_t \Delta_\phi f ] \phi_t
\frac{\omega_\phi^n}{n!} \nonumber \\
&=
\frac{1}{2}\int_M
(C[f] * \p  \phi_t , \, \bar \p \phi_t )_{g}
\frac{\omega_\phi^n}{n!},
\end{align}
where the argument behind the first equality is exactly the same as in formula \eqref{parta}.
We plug \eqref{equ_first_add_of_I} and \eqref{equ_second_add_of_I}
in \eqref{equ_term_covxity_k_energy} and we get
\begin{align*}
 &I = \int_M
[
(\Delta_\phi \phi_t)^2 - ( (\rho_{\phi} - \p \bar \p f) *\p  \phi_t , \, \bar \p \phi_t)_{g_\phi}
] \frac{\omega_\phi^n}{n!}\\
&\phantom{+++}+\frac{1}{2}\int_M [
-(\p \bar \p f *\p  \phi_t , \, \bar \p \phi_t )_{g_\phi}
+ (S - \underline{S}) (\p \phi_t , \,  \bar \p \phi_t )_{g_\phi}
]
\frac{\omega_\phi^n}{n!}\\
&=\int_M
[
(\Delta_\phi \phi_t)^2 - ( \rho_{\phi} *\p  \phi_t , \, \bar \p \phi_t )_{g_\phi}
] \frac{\omega_\phi^n}{n!}\\
&\phantom{+++}+
\frac{1}{2}\int_M
[
(\p \bar \p f *\p  \phi_t , \, \bar \p  \phi_t )_{g_\phi}
+ (S - \underline{S}) (\p \phi_t , \,  \bar \p \phi_t )_{g_\phi}
]
\frac{\omega_\phi^n}{n!}\\
&=
\frac{1}{2}\int_M \{ 2|D \phi_t|_{g_\phi}^2 +
([
\p \bar \p f + (S - \underline{S})g_{\phi}
]
*\p \phi_t ,\, \bar \p \phi_t)_{g_\phi}
 \} \frac{\omega_\phi^n}{n!},
\end{align*}
where at the last equality we employed the Ricci identity and we wrote $D \phi_t$
which stands, in a coordinate chart,
for $g^{\a \bar \lambda}\frac{\p^2 \phi_t}{\p z^{\bar \lambda} \p z^{\bar \b}} \frac{\p}{\p z^\a} \otimes dz^{\bar \b}$
(see \cite[page 100, c)]{Calabi2}).
In particular, the above formula shows that, if $d\phi (t_0)$ induces
a constant scalar curvature K\"ahler metric, that is $S(\phi(t_0)) = \underline{S}$
and $f( \phi (t_0)) = 0 $, then
\[
 \frac{\partial^2}{\partial t^2}\nu(\phi) (t_0) \geq 0,
\]
which completes the proof of the proposition.
 \end{proof}
\smallskip
\begin{rmk}
It is pointed out by Xiu Xiong Chen and the second author  (see \cite{ChenZheng}, Remark $3.2$) that the Pseudo-Calabi flow
is the gradient flow of the K-energy for the Dirichlet metric.
Here, we want to notice that also the (normalized) K\"ahler Ricci flow is the gradient flow of the
K-energy in the first Chern class for the Dirichlet metric.
Indeed, the first variation of the K energy is, just from the very definition,
\[
 \frac{\partial \nu }{\partial t} (d\phi) = \int_M ( d\phi_t , df  )_{g_\phi} \frac{\omega_\phi^n}{n!}.
\]
Moreover the normalized K\"ahler Ricci flow at the potentials level is
\[
 \frac{\partial}{\partial t} \phi = \log \left( \frac{\omega_\phi^n}{\omega^n} \right) +\phi + h_\omega + c(t).
\]
Then, taking $\partial \overline{\partial}$ on both sides and using the Maximum Principle,
we get the claimed conclusion.
\end{rmk}
\medskip
\section{The family of weighted metrics}\label{section the family of weighted metrics}
Here we introduce a family of Riemannian structures, whose distinguished element is the
Calabi metric.

\begin{defn}
 We consider the following family of Riemannian structures on the space $\mathcal{C}$,
 which we label as \emph{weighted metrics};
 let be given  $F\in \mathcal{C}$ and $G\in T_F\mathcal{C}$, and let
 $\chi \in C^{\infty}(]0,\, +\infty[,\, ]0,\, +\infty[)$. Then we define
 \begin{align}\label{eqn:defined scalar product}
  <G_1 ,\, G_2 >_{F;\, \chi} : = \int_M \chi(F) G_1 G_2  \frac{\omega^n}{n!} \;.
 \end{align}
We call the function $\chi$ the \emph{weight} of the Riemannian product defined above.
Any such Riemannian product is smooth
(a role here is played by the smoothness of the weight); moreover, any such Riemannian
product is non-degenerate (since the weight is a positive function).
\end{defn}

\begin{rmk}\label{rem:scalar product at level of kahler metrics}
The space of K\"ahler metrics  $\mathcal{H}$
together with the following Riemannian product
is isometric to the space $\mathcal{C}$ with
the Riemannian product \eqref{eqn:defined scalar product};
given $\omega_\phi \in \mathcal{H}$,
 and $\psi \in T_{\omega_\phi}\mathcal{H}$, that product is
 given by
 \begin{align*}
  <\psi_1 , \psi_2 >_{\omega_\phi ; \chi} = \int_M
  \chi \left(\frac{\omega_\phi^n}{\omega^n} \right )
  \left( \Delta_\phi \psi_1 \right) \cdot \left( \Delta_\phi \psi_2 \right)
  \left( \frac{\omega_\phi^n}{\omega^n}\right)^2
  \frac{\omega^n}{n!} \; .
 \end{align*}
\end{rmk}

\begin{rmk}
 A special sub-family of Riemannian products is got by specifying the weight to be of
 the form $\chi (x) = x^k$, for some $k\in \mathbb{Z}$. In this case, we are considering
 the metric, for a fixed $k\in \mathbb{Z}$,
\begin{align*}
 <G_1 , G_2 >_F := \int_M G_1 G_2 F^k \omega^n \; .
\end{align*}
The Calabi metric is retrieved  for $k=-1$: that is the weight that gives the Calabi
metric is $\chi (x) = x^{-1}$. According to Remark
\ref{rem:scalar product at level of kahler metrics},
we also retrieve its expression at the level of the space of K\"ahler metrics, which is
\begin{align*}
 <\psi_1 , \psi_2 >_{\omega_\phi ; x^{-1}} =  \int_M \left( \Delta_\phi \psi_1 \right) \cdot
 \left( \Delta_\phi \psi_2 \right)
 \frac{\omega_\phi^n}{n!}\; ,
\end{align*}
(compare with the expression in \cite{Calamai}).
\end{rmk}

The following sections concerns the computation of formulas and equations
that come from the definition of weighted metrics \eqref{eqn:defined scalar product}.

\subsection{The Levi-Civita connection}

We have the existence of the Levi-Civita covariant derivative for
any of the Riemannian structures defined by \eqref{eqn:defined scalar product}.
For the precise definition of the Levi-Civita covariant derivative, we refer to
\cite[Definition 7]{Calamai}

\begin{prop}\label{prop:levicivita existence and uniqueness}
The Levi-Civita covariant derivative for the Riemannian structure
\eqref{eqn:defined scalar product} exists and is unique.
 Let $F=F(t)$ be a smooth path in $\mathcal{C}$ and let $G=G(t)$
 a vector field on the path; then the Levi-Civita covariant derivative of
 $G$ along $F$ is given by
 \begin{align}\label{eqn: formula of the covariant derivative}
  D^{\chi}_t G = \frac{dG}{dt}
  + \frac{\chi' (F)}{2\chi(F)}\frac{dF}{dt} G
  -\frac{\chi^{-1} (F)}{\int_M \chi^{-1} (F)\frac{\omega^n}{n!}}
  \int_M \frac{\chi' (F)}{2\chi(F)}\frac{dF}{dt} G \frac{\omega^n}{n!}\; ,
 \end{align}
 where the inverse of $\chi$ is understood with respect
 to the multiplicative structure, i.e. $\chi^{-1}= \frac{1}{\chi}$.
\end{prop}
\begin{proof}
 The additivity, the homogeneity and the Leibniz rule
 of \eqref{eqn: formula of the covariant derivative}
 are evident. Concerning the compatibility with the
 metric, we compute
 \begin{align*}
  \frac{d}{dt} <G ,G >_{F; \chi}&= \frac{d}{dt}\int_M \chi(F)G^2 \frac{\omega^n}{n!}\\
  &= 2\int_M G \chi(F) \left(
  \frac{dG}{dt} + \frac{\chi'(F)}{2\chi(F)}\frac{dF}{dt} G
                       \right) \frac{\omega^n}{n!}\; .
 \end{align*}
The trick now is that we can add, inside the round brackets, a term of the form
$\lambda \chi^{-1}(F)$, where $\lambda$ is any real constant. This will not change the
value of the integral, since we have the constraint $\int_M G \frac{\omega^n}{n!}$
on tangent vectors. Thus, we can write
\begin{align*}
  \frac{d}{dt} <G ,G >_{F; \chi} = 2\int_M G \chi(F) \left(
  \frac{dG}{dt} + \frac{\chi'(F)}{2\chi(F)}\frac{dF}{dt} G
  +\lambda \chi^{-1}(F)
                       \right) \frac{\omega^n}{n!}\; .
 \end{align*}
Now, the constant $\lambda$ is chosen in order to have a tangent vector insider the
round brackets; namely, we want the following constraint to be fulfilled
\begin{align*}
 \int_M \left(
  \frac{dG}{dt} + \frac{\chi'(F)}{2\chi(F)}\frac{dF}{dt} G
  +\lambda \chi^{-1}(F)
                       \right) \frac{\omega^n}{n!} = 0\; .
\end{align*}
The above equality holds if and only if $\lambda$ is
\begin{align*}
 \lambda =  -\frac{1}{\int_M \chi^{-1} (F)\frac{\omega^n}{n!}}
  \int_M \frac{\chi' (F)}{2\chi(F)}\frac{dF}{dt} G \frac{\omega^n}{n!}\; ,
\end{align*}
according to the statement. Finally, to verify the property on torsion free,
we let $F=F(s,t)$ be a two parameter family of smooth curves,
then we take  $G=\frac{dF}{ds}$, and we notice that the equation
\eqref{eqn: formula of the covariant derivative} is symmetric in $s$ and $t$.
This completes the proof of the proposition.
\end{proof}
\smallskip
\subsection{The Sectional Curvature}
We prepare the main result of this section via some preliminary lemmas.
In this section, let $F=F(s,t)$ denote a two-parameter family of smooth curves on $\mathcal{C}$.
Also, let $G=G(s,t), H=H(s,t)\, \in T_{F(s,t)}\mathcal{C}$ be smooth curves of tangent vectors.
\begin{lem}\label{lem: formula for the Riemann curvature tensor}
 The following formula holds
 \begin{align*}
  &D^{\chi}_s D^{\chi}_t G - D^{\chi}_t D^{\chi}_s G \\
  &=\frac{1}{4\int_M \chi^{-1} \frac{\omega^n}{n!}} \frac{\chi'(F)}{\chi^2(F)}F_s
  \int_M \frac{\chi ' (F)}{\chi(F)} F_t G \frac{\omega^n}{n}
  -
  \frac{1}{4\int_M \chi^{-1} \frac{\omega^n}{n!}} \frac{\chi'(F)}{\chi^2(F)}F_t
  \int_M \frac{\chi ' (F)}{\chi(F)} F_s G \frac{\omega^n}{n}\\
  &+
  \frac{\chi^{-1}(F)}{2}\int_M \frac{\chi'(F)}{\chi^2(F)}F_s \frac{\omega^n}{n!}
  \int_M \frac{\chi'(F)}{\chi(F)}F_t G \frac{\omega^n}{n!}
  -
   \frac{\chi^{-1}(F)}{2}\int_M \frac{\chi'(F)}{\chi^2(F)}F_t \frac{\omega^n}{n!}
  \int_M \frac{\chi'(F)}{\chi(F)}F_s G \frac{\omega^n}{n!}\\
  &+
  \frac{\chi^{-1}(F)}{4}\left(\int_M \chi^{-1}(F)\frac{\omega^n}{n!}\right)^2
  \int_M \frac{\chi'(F)}{\chi^2(F)}F_s \frac{\omega^n}{n!}
  \int_M \frac{\chi'(F)}{\chi(F)}F_t G \frac{\omega^n}{n!}\\
  &-
  \frac{\chi^{-1}(F)}{4}\left(\int_M \chi^{-1}(F)\frac{\omega^n}{n!}\right)^2
  \int_M \frac{\chi'(F)}{\chi^2(F)}F_t \frac{\omega^n}{n!}
  \int_M \frac{\chi'(F)}{\chi(F)}F_s G \frac{\omega^n}{n!}\; .
   \end{align*}
\end{lem}
\begin{proof}
 The claimed formula is lengthy computation which just makes use of
 \eqref{eqn: formula of the covariant derivative}.
\end{proof}

\begin{rmk}
 In the case of the Calabi metric, the last four addenda vanish.
 In fact, when $\chi(x) = x^{-1}, \, x>0$,  then the ratio $\frac{\chi '(F)}{\chi^2 (F)} = -1$,
 and so the integral
 \begin{align*}
  \int_M \frac{\chi'(F)}{\chi^2(F)}F_{\bullet} \frac{\omega^n}{n!}\; ,
 \end{align*}
 appearing in all of the four addenda of the formula, vanishes since any tangent vector
 $F_{\bullet}$ has to have mean zero by \eqref{eqn:tangent vectors}.
\end{rmk}

\begin{lem}\label{lemma: preparation 2 for sectional curvature weighted}
 The following formula holds
 \begin{align*}
  &<D^{\chi}_s D^{\chi}_t G- D^{\chi}_t D^{\chi}_s G, H >_{F ; \chi} \\
  &=\frac{1}{4\int_M \chi^{-1} \frac{\omega^n}{n!}}
  \int_M \frac{\chi'(F)}{\chi(F)}F_s H \frac{\omega^n}{n}
  \int_M \frac{\chi ' (F)}{\chi(F)} F_t G \frac{\omega^n}{n}\\
  &-
  \frac{1}{4\int_M \chi^{-1} \frac{\omega^n}{n!}}
  \int_M \frac{\chi'(F)}{\chi(F)}F_t H \frac{\omega^n}{n}
  \int_M \frac{\chi ' (F)}{\chi(F)} F_s G \frac{\omega^n}{n}\; .
 \end{align*}
\end{lem}
\begin{proof}
 The formula follows from a  straightforward computation that builds on
 the symmetries given by the Lemma \ref{lem: formula for the Riemann curvature tensor}.
\end{proof}
We are now ready for the main result of this section
\begin{thm}\label{thm: formula of sectional curvature}
 The sectional curvature of a two plane spanned by $F_s$ and $F_t$
 has the following expression
 \begin{align}\label{eqn:formula sectional curvature}
  &K^{\chi}[F](F_s , F_t ) \\
 &= \frac{1}{4\int_M \chi^{-1} \frac{\omega^n}{n!}}
  \frac{\left(\int_M \frac{\chi'(F)}{\chi(F)}(F_s)^2 \frac{\omega^n}{n!}\right)
  \left(\int_M \frac{\chi'(F)}{\chi(F)} (F_t)^2 \frac{\omega^n}{n!}\right)
  -\left(\int_M  \frac{\chi'(F)}{\chi(F)} F_s F_t \frac{\omega^n}{n!} \right)^2}{\left(\int_M \chi(F) (F_s)^2 \frac{\omega^n}{n!}\right)
  \left(\int_M \chi(F) (F_t)^2 \frac{\omega^n}{n!}\right)
  -\left(\int_M  \chi(F) F_s F_t \frac{\omega^n}{n!} \right)^2}  \nonumber\; .
 \end{align}
 \end{thm}
 \begin{proof}
 It follows directly from Lemma \ref{lemma: preparation 2 for sectional curvature weighted}.
 \end{proof}

 \begin{cor}
 For every weight function $\chi$,
 then the space $\mathcal{C}$ always has non-negative sectional curvature.
 If $\chi$ is constant, then the right-hand side of
 \eqref{eqn:formula sectional curvature} is clearly zero.
\end{cor}

\begin{rmk}
 It can be immediately read off equation \eqref{eqn:formula sectional curvature} that
 the Calabi metric gives a constant scalar curvature of value $\frac{1}{4V}$.
\end{rmk}

\begin{rmk}
When the function  $\chi^{-1}$  has convexity or concavity properties,
then the Jensen inequality provides an estimate
for the factor $\int_M \chi^{-1}(F) \frac{\omega^n}{n!}$.
\end{rmk}

\subsection{The Geodesic equation}
The aim of the present subsection is to write down the geodesic equation for each one of the
weighted metrics. We will solve the equation in a special case, and also we remark here
that the Calabi metric case was already considered. In all the remainder cases, it would be interesting to know
whether there are smooth solutions or not to the geodesic equation.
\smallskip
\begin{prop}\label{prop: geodesic equation}
The geodesic equation is
 \begin{align}\label{eqn: geodesic equation}
  D^{\chi}_t \frac{dF}{dt} = \frac{d^2 F}{dt^2}
  + \frac{\chi' (F)}{2\chi(F)}\left(\frac{dF}{dt} \right)^2
  -\frac{\chi^{-1} (F)}{\int_M \chi^{-1} (F)\frac{\omega^n}{n!}}
  \int_M \frac{\chi' (F)}{2\chi(F)}\left(\frac{dF}{dt} \right)^2 \frac{\omega^n}{n!}=0\; .
\end{align}
\end{prop}
\smallskip
\begin{rmk}
We would like to rewrite an equivalent form of \eqref{eqn: geodesic equation}
with respect to the K\"ahler potential $\phi$. So, we multiply by a factor $F^{-1}$,
and using that $F= \frac{\omega_\phi^n}{\omega^n}$, we have
\begin{align} \label{eqn: geodesic equation wrt kaehler potentials}
   &\frac{d}{dt}(\Delta_\phi \dot \phi)
   + (\Delta_\phi \dot \phi)^2
   + \frac{\omega_\phi^n}{\omega^n}
   \frac{\chi' (\frac{\omega_\phi^n}{\omega^n})}{2\chi(\frac{\omega_\phi^n}{\omega^n})}
   \left(\Delta_\phi \dot \phi \right)^2 \\
   &\phantom{+++}-\frac{\omega^n}{\omega_\phi^n}
   \frac{\chi^{-1} (\frac{\omega_\phi^n}{\omega^n})}{\int_M \chi^{-1} (\frac{\omega_\phi^n}{\omega^n})\frac{\omega^n}{n!}}
  \int_M
  \frac{\chi' (\frac{\omega_\phi^n}{\omega^n})}{2\chi(\frac{\omega_\phi^n}{\omega^n})}
  \left(\frac{\omega_\phi^n}{\omega^n} \Delta_\phi \dot \phi \right)^2 \frac{\omega^n}{n!}=0 \; .\nonumber
\end{align}
\end{rmk}
\smallskip
\begin{rmk}
The case of the Calabi metric can be retrieved in the following way;
when $\chi(x)= x^{-1}$, then the ratio $\frac{\chi'(F)}{\chi(F)} = -F^{-1}$,
so that the equation \eqref{eqn: geodesic equation} reads
\begin{align}\label{eqn:retrieving the calabi geodesic}
    \frac{d^2 F}{dt^2}
    - \frac{F^{-1}}{2}\left(\frac{dF}{dt} \right)^2
    +\frac{F}{2V}
    \int_M \frac{ 1}{F}\left(\frac{dF}{dt} \right)^2 \frac{\omega^n}{n!}=0\; .
\end{align}
 Then, using that
 $2F\frac{d^2 F}{dt^2} -\left(\frac{dF}{dt} \right)^2  =  4F^{\frac{3}{2}} \frac{d^2 \sqrt{ F}}{dt^2}$,
  we write the equation \eqref{eqn:retrieving the calabi geodesic} as
\begin{align*}
\frac{d^2 \sqrt{ F}}{dt^2}+ \frac{\sqrt{F}}{4V}
\int_M \frac{ 1}{F}\left(\frac{dF}{dt} \right)^2 \frac{\omega^n}{n!}=0\; .
    \end{align*}
   The key ingredient now is to recognize that the integral
   $\int_M \frac{ 1}{F}\left(\frac{dF}{dt} \right)^2 \frac{\omega^n}{n!}$
   is a constant of motion. Then the Calabi geodesic are solutions of the
   harmonic motion equation, and they make sense as long as $F$ is positive.
\end{rmk}

\subsection{Convexity of the $K$-energy along geodesics}
In this section we address the question on the convexity of
the Mabuchi $K$-energy along the geodesics
of some scalar product of the family \eqref{eqn:defined scalar product},
provided the existence of such geodesics. Thus, at the present level, the
discussion is formal.
By the way, we will compute an interesting example for which the
computation will be substantial.
\smallskip
\begin{prop}\label{prop:sign of second variation}
For a generic element of the family of scalar products \eqref{eqn:defined scalar product},
the second variation of the Mabuchi energy along its geodesic (provided its existence)
is given by the following formula
\begin{align*}
&\frac{d^2 \nu}{dt^2}
=\int_M (\Delta_\phi \phi ')^2 \frac{\omega_\phi^n}{n!} \\
&+ \int_M
\left[
 - \frac{\omega_\phi^n}{\omega^n}
   \frac{\chi' (\frac{\omega_\phi^n}{\omega^n})}{2\chi(\frac{\omega_\phi^n}{\omega^n})}
   \left(\Delta_\phi  \phi ' \right)^2
   +\frac{\omega^n}{\omega_\phi^n}
   \frac{\chi^{-1} (\frac{\omega_\phi^n}{\omega^n})}{\int_M \chi^{-1} (\frac{\omega_\phi^n}{\omega^n})\frac{\omega^n}{n!}}
  \int_M
  \frac{\chi' (\frac{\omega_\phi^n}{\omega^n})}{2\chi(\frac{\omega_\phi^n}{\omega^n})}
  \left(\frac{\omega_\phi^n}{\omega^n} \Delta_\phi \phi' \right)^2 \frac{\omega^n}{n!}
\right]\\
&\cdot
\left(
 \log\frac{\omega_\phi^n}{\omega_0^n}  - P
\right)
\frac{\omega_\phi^n}{n!}
+\int_M
([\p \bar \p P-\rho(\omega_0)]* \p (\phi ') , \, \bar \p \phi' )_{g_\phi}\frac{\omega_\phi^n}{n!}.
\end{align*}
Here above $\triangle_\phi P:=\omega_{\phi}* \rho(\omega_0)-\underline S$,
and $\rho(\omega_0) = \rho (0)$ is the
Ricci form corresponding to the K\"ahler form $\omega_\phi$.
\end{prop}
\begin{proof}
\begin{align*}
&\frac{d^2 \nu}{dt^2}
= -\frac{d}{dt}\int_M \phi ' (S(\phi)-\underline{S}) \frac{\omega_\phi^n}{n!}
=-\frac{d}{dt}\int_M \phi ' \omega_\phi * \rho(\phi)
\frac{\omega_\phi^n}{n!}\\
&=-\frac{d}{dt}\int_M \phi ' \omega_\phi * (\rho(\phi)-\rho(0))
\frac{\omega_\phi^n}{n!}
-\frac{d}{dt}\int_M \phi ' \omega_\phi * \rho(0) \frac{\omega_\phi^n}{n!}.
\end{align*}
The first term is
\begin{align*}
&=\frac{d}{dt}\int_M \phi '
\left(
\Delta_\phi \log\frac{\omega_\phi^n}{\omega_0^n}
\right)
\frac{\omega_\phi^n}{n!} = \frac{d}{dt}\int_M \Delta_\phi \phi '
\left(
 \log\frac{\omega_\phi^n}{\omega_0^n}
\right)
\frac{\omega_\phi^n}{n!}
\\
&= \int_M
\left(
\frac{d}{dt}\Delta_\phi \phi '
\right)
\left(
 \log\frac{\omega_\phi^n}{\omega_0^n}
\right)
\frac{\omega_\phi^n}{n!}
+\int_M (\Delta_\phi \phi ')^2
\frac{\omega_\phi^n}{n!}
+\int_M \Delta_\phi \phi '
\left(
 \log\frac{\omega_\phi^n}{\omega_0^n}
\right)
(\Delta_\phi \phi ' )
\frac{\omega_\phi^n}{n!}\\
&=\int_M (\Delta_\phi \phi ')^2 \frac{\omega_\phi^n}{n!}
+ \int_M
\left[
\frac{d}{dt}\left(\Delta_\phi \phi '\right)
+ \left( \Delta_\phi \phi '\right)^2
\right]
\left(
 \log\frac{\omega_\phi^n}{\omega_0^n}
\right)
\frac{\omega_\phi^n}{n!}
\; .\\
\end{align*}
The second term is
\begin{align*}
-\frac{d}{dt}\int_M \phi ' \omega_\phi * \rho(0) \frac{\omega_\phi^n}{n!}
=-\frac{d}{dt}\int_M \phi '[ \omega_\phi * \rho(0)-\underline S] \frac{\omega_\phi^n}{n!}
\end{align*}
\begin{align*}
=-\frac{d}{dt}\int_M \phi ' \triangle_\phi P\frac{\omega_\phi^n}{n!}
=-\frac{d}{dt}\int_M \left( \triangle_\phi \phi ' \right) P\frac{\omega_\phi^n}{n!}
\end{align*}
\begin{align*}
=-\int_M \left[ \frac{d}{dt}\left( \triangle_\phi \phi ' \right)
+ \left( \triangle_\phi \phi ' \right)^2 \right] P\frac{\omega_\phi^n}{n!}
-\int_M \phi ' \triangle_\phi P' \frac{\omega_\phi^n}{n!}
\end{align*}
\begin{align*}
=-\int_M \left[ \frac{d}{dt}\left( \triangle_\phi \phi ' \right)
+ \left( \triangle_\phi \phi ' \right)^2 \right] P\frac{\omega_\phi^n}{n!}
+\int_M  ([\p \bar \p P- \rho(\omega_0)]*\p \phi' , \, \bar \p \phi ')_{g_\phi} \frac{\omega_\phi^n}{n!}.
\end{align*}
So, we found that
\begin{align*}
 &\frac{d^2 \nu}{dt^2}
=\int_M (\Delta_\phi \phi ')^2 \frac{\omega_\phi^n}{n!} \\
&+ \int_M
\left[
\frac{d}{dt}(\Delta_\phi  \phi') + (\Delta_\phi  \phi')^2
 \right]
\left(
 \log\frac{\omega_\phi^n}{\omega_0^n}  - P
\right)
\frac{\omega_\phi^n}{n!}\\
&+\int_M  ([\p \bar \p P- \rho(\omega_0)]*\p \phi' , \, \bar \p \phi ')_{g_\phi} \frac{\omega_\phi^n}{n!}\, ,
\end{align*}
and plugging in \eqref{eqn: geodesic equation wrt kaehler potentials}, we complete
the proof of the proposition.
\end{proof}

We want to give an example of a scalar product of the family
\eqref{eqn:defined scalar product} such that the Mabuchi $K$-energy is
convex along its geodesic, at least in some special cases.
\subsection{Entropy function}
Recall the explicit form
of $K$-energy in Chen \cite{MR1772078} and Tian \cite{MR1787650},
\begin{align*}
\nu_\om(\vphi)
&=\frac{1}{V}\int_M\log\frac{\om^n_\vphi}{\om^n}\om_\vphi^{n}
+\frac{\as}{V}\sum_{i=0}^n\frac{n!}{(i+1)!(n-i)!}
\int_{M}\vphi\om^{n-i}\wedge(\p\bar\p\vphi)^i\\
&-\frac{1}{V}\sum_{i=0}^{n-1}\frac{n!}{(i+1)!(n-i-1)!}\int_{M}\vphi
Ric\wedge\om^{n-1-i}\wedge(\p\bar\p\vphi)^{i}.
\end{align*}
Here, we consider the leading term of the $K$-energy.
\begin{defn}
We introduce the entropy function in the the space of volume conformal factors $\mathcal C$ (see \eqref{conformal})
\begin{align}\label{entropy}
\mathcal E(F)= \int_M F \log (F) \frac{\omega^n}{n!} .
\end{align}
In which, $F$ is the volume conformal factor $\frac{\om^n_\vphi}{\om^n}$.
\end{defn}
We compute that along a differential curve $F(t)$,
$$\frac{d}{dt} \mathcal E(F(t))=\int_M F' \log (F) \frac{\omega^n}{n!}+V.$$
Since the volume of the K\"ahler metrics in a fixed K\"ahler class is a topological constant, we obtain the following second order derivative
\begin{align}\label{eqn: expression for intermediate example}
\frac{d^2}{dt^2} \mathcal E(F(t))=\frac{d}{dt} \int_M F' \log (F) \frac{\omega^n}{n!} .
\end{align}
Now assuming that $F(t)$ is a geodesic and using \eqref{eqn: geodesic equation}, we get
\begin{align*}
& \frac{d}{dt} \int_M F' \log (F) \frac{\omega^n}{n!}
=\int_M F '' \log (F)\frac{\omega^n}{n!}+ \int_M \frac{(F')^2}{F} \frac{\omega^n}{n!}\\
&=- \frac{1}{2}\int_M (F ')^2 \rho '(F) \log (F)\frac{\omega^n}{n!}\\
&+ \frac{1}{2}
\frac{\int_M \chi^{-1}(F)\log(F)\frac{\omega^n}{n!}}{\int_M \chi^{-1}(F)\frac{\omega^n}{n!}}
\int_M (F ')^2 \rho ' (F) \frac{\omega^n}{n!}
+ \int_M \frac{(F')^2}{F} \frac{\omega^n}{n!}\;,
\end{align*}
where $\rho=\log\chi$.
For now, we are able to find and example of such a $\rho$ for which the sign of the
expression \eqref{eqn: expression for intermediate example} is non negative. It suffices to choose $$\rho(x)=\log\log (x+1).$$ Then for any $x>0$, $\rho'=\frac{1}{(x+1) \log (x+1)}>0$ and $\chi=\log (x+1)>0$. So
 \begin{align}\label{osservazione per ocnvessita}
  \frac{1}{x} - \frac{1}{2}\rho '(x) \log x=
    \frac{1}{x} -\frac{1}{2} \frac{\log x}{(x+1) \log (x+1)}  > 0 \;.
 \end{align}
Therefore, $\mathcal E(F)$ is convex along the geodesic with respect to the weighted metric with weight $\chi=e^\rho$.
\section{The example of the constant weight metric}
In this section we consider back the
features listed in  Section \ref{section the family of weighted metrics} for the
special case given by the metric which has weight constantly one;
namely, the metric gotten by plugging $\chi = 1$
into \eqref{eqn:defined scalar product}, that is
\begin{align}\label{eqn: defn flat metric}
<G_1 ,\, G_2 >_{F;\, 1} : = \int_M G_1 G_2  \frac{\omega^n}{n!} \;.
\end{align}

\begin{rmk}
At the level of K\"ahler metrics,
 the scalar product is given by
 \begin{align*}
  \int_M (\Delta_{\phi} \psi )^2
\frac{\omega_{\phi}^n}{\omega^n}
\frac{\omega_\phi^n}{n!}\; .
 \end{align*}
 \end{rmk}

\subsection{The Levi-Civita connection for the constant weight metric}

The existence and uniqueness of the Levi-Civita connection
for the family of weighted metrics \eqref{eqn:defined scalar product}
has already been proved and discussed. So, as an application of Proposition
\ref{prop:levicivita existence and uniqueness}, we have

\begin{cor}
 The formula for the Levi-Civita connection for the constant weight metric
 is given by
\begin{align}\label{eqn: covariant derivative for the dlat metric}
 D_t^1 G = \frac{d}{dt}G \; .
\end{align}
\end{cor}

\subsection{The Sectional Curvature for the constant weight metric}
Concerning the sectional curvature, we can read off Theorem
\ref{thm: formula of sectional curvature} that it is constantly zero.
Moreover, Lemma \ref{lem: formula for the Riemann curvature tensor}
tells us that also the Riemann curvature tensor is constantly zero.
Let us state this observation precisely.
\begin{cor}\label{cor: sectional curvature of constante weight metric}
 The constant weight metric \eqref{eqn: defn flat metric} has sectional curvature zero.
\end{cor}

\subsection{Geodesics for the constant weight metric}
We presently don't know whether most of the weighted metrics do admit geodesics.
We know for sure that it is the case for the Calabi metric, and now we want to show
that a geometry as rich and  explicit as that one of the Calabi metric exists also for
the constant weight metric.
We begin with applying Proposition \ref{prop: geodesic equation}
to the constant weight metric.
\begin{cor}
 The geodesic equation for the flat scalar product \eqref{eqn: defn flat metric}
 is  $F_{tt} = 0$.
 In terms of K\"ahler potentials, the geodesic equation is
 \begin{align}\label{eqn: geodesic eqn flat metric at kaehler potentials level}
  \frac{d}{dt}\left( \Delta_\phi \phi '\right) + \left(\Delta_\phi \phi ' \right)^2 = 0
  \; .
 \end{align}
\end{cor}
The geodesic equation for the constant weight metric is thus solvable,
both when we consider the Cauchy and the Dirichlet problems.
Moreover, we do have smooth  explicit geodesics.
We summarize those properties in the next result.
\begin{thm}\label{thm: geodesics for the constant weight metric}
 For any given $F_0 , \, F_1 \in \mathcal{C}$, there is a unique geodesic
 arc for the constant weight metric \eqref{eqn: defn flat metric} connecting
 $F_0$ and $F_1$. Such geodesic arc is given by
 \begin{align}\label{eqn: geodesic arc for the constant weight metric}
  F(t)= t F_1 + (1-t) F_0 \in \mathcal{C}, \qquad t\in [0,\, 1]  \; .
 \end{align}
For any $F\in \mathcal{C}$ and any $G\in T_F\mathcal{C}$, the geodesic line starting at $F$
with speed $G$ is given by
\begin{align}\label{eqn: geodesic line for the constant weight metric}
 F(t)=  F + tG,
\end{align}
where $t$ ranges in a suitable open neighborhood of $0$ such that $F(t)>0$.
\end{thm}
We say $G$ is the non-negative direction if $G\geq 0$ and negative if $G< 0$.
Note that the geodesic line may be degenerate along the negative directions.

\subsection{Geometric properties of the constant weight metric}
In this subsection, we always assume that when $C_1(M) \leq 0$.
A surprising fact  that we will discuss later is a
 convexity property of the $K$-energy along the geodesics
 of the space $(\mathcal{C}, < \cdot ,\, \cdot >_{F;\, 1})$ in the cases when $C_1(M) \leq 0$.
As a result, we find back the uniqueness of K\"ahler-Einstein metrics,
in any K\"ahler class when the first Chern class vanishes, and
in $-2\pi C_1 (M)$ when the first Chern class is negative. We have to emphasize here
that the uniqueness of K\"ahler-Einstein metrics when
$C_1 (M) \leq 0$ was proved by Calabi by integration by part \cite{MR0085583}.


\begin{prop}\label{prop:sign of second variation}
Along the geodesics of the constant weight metric,
the second variation of the Mabuchi energy is given by the following formula
\begin{align*}
\frac{d^2 \nu}{dt^2}
&=\int_M (\Delta_\phi \phi ')^2 \frac{\omega_\phi^n}{n!}
+\int_M  ([\p \bar \p P - \rho(\omega_0)]* \p \phi ' , \bar \p \phi ')_{g_{\phi}} \frac{\omega_\phi^n}{n!}\; .
\end{align*}
Here above $\triangle_\phi P:=\omega_\phi * \rho(0)-\underline S$.
\end{prop}
\begin{proof}
The argument goes similarly to that one of Proposition \ref{prop:sign of second variation}.
\end{proof}
The Ricci potential $h_\om$ is defined as
\begin{align*}
\rho(\om)-\l\om=\frac{\sqrt{-1}}{2}\p\bar\p
h_\om.
\end{align*}
Under the normalized conditions $\int_Me^{h_\om}\om^n=V$.
\begin{lem}\label{lem belongs to the first Chern class}
When $\om_\phi$ belongs to the first Chern class, then we have
\begin{align*}
P=-\l\phi+h_\om
\end{align*}
\end{lem}
\begin{proof}
Since
\begin{align*}
S-\l n
&=g_\phi^{i\bar j}(\rho_{i\bar j}(\om_\phi)-\l
g_{\phi i\bar j})\\
&=g_\phi^{i\bar j}(\rho_{i\bar j}(\om_\phi)-\rho_{i\bar j}(\om)+\rho_{i\bar
j}(\om)-\l g_{i\bar j}+\l g_{i\bar j}-\l
g_{\phi i\bar
j})\\
&=\tri_\phi(-\log\frac{\om^n_\phi}{\om^n}-\l\phi+h_\om).
\end{align*}
Therefore we have
\begin{align}
P=-\l\phi+h_\om.
\end{align}
\end{proof}
After applying Lemma \ref{lem belongs to the first Chern class}
into Proposition \ref{prop:sign of second variation}, we obtain
\begin{align*}
 &\frac{d^2 \nu}{dt^2}
=
\int_M (\Delta_\phi \phi ')^2 \frac{\omega_\phi^n}{n!}
 - \lambda \int_M |\nabla \phi' |_{g_\phi}^2 \frac{\omega_\phi^n}{n!} \; .
\end{align*}
Suppose that $C_1 (M) \leq 0$, then $\frac{d^2 \nu}{dt^2}\geq 0$. We obtain that
\begin{prop}
When $C_1 (M) \leq 0$, the $K$-energy is convex along geodesics of the constant weight metric.
\end{prop}
A direct corollary of the convexity is $$\frac{d}{dt}\nu(\om_\vphi)\vert_{t=1}\geq\frac{d}{dt}\nu(\om_\vphi)\vert_{t=0}.$$
We assume that $\omega_0$ and $\omega_1 = \omega_\phi$ are two K\"ahler-Einstein metrics in the same K\"ahler class.
Then, write $\omega_1 = \omega_0 + i\partial \overline{\partial}\psi$. Since $\int_M \phi ' \omega_\phi^n =0$, the above equation implies that $\phi ' = 0$, and thus the
geodesic is trivial and hence $\omega_0 = \omega_1$.
Hence, we arrive at the alternative proof of Calabi's uniqueness theorem. The energy and
the distance from $\om$ to $\om_\vphi$ w.r.t. the constant weight metric are
\begin{align*}
E(\om,\om_\vphi)&=\int_M(\tri_\vphi\vphi')^2(\frac{\om^n_\vphi}{\om^n})^2\om^n,\\
d(\om,\om_\vphi)&=\int_0^1\sqrt{E(\om,\om_\vphi)}dt.
\end{align*}
\begin{thm}\label{unike}
When $C_1 (M) \leq 0$, the K\"ahler-Einstein metric is unique.
\end{thm}
\begin{defn}
We introduce a new functional which is the square norm of the scalar potential $f$ under the normalization condition $\int_M e^f\frac{\omega_\vphi^n}{n!}=1$,
\begin{align*}
G(\om_\vphi)=\int_M f^2\frac{\omega^n}{n!}.
\end{align*}
\end{defn}
In particular, in the canonical class, $f$ is the Ricci potential $f=-\log\frac{\om^n_\vphi}{\om^n}+h_\om-\l\vphi$. We list the properties of $G$ in the canonical class.
\begin{enumerate}
  \item The Euler-Lagrange equation is $(\tri_\vphi+\l)(f\frac{\om^n}{\om^n_\vphi}).$ When $\l\leq0$, $f=0$. So the critical points are the K\"ahler-Einstein metrics.
  \item At a K\"ahler-Einstein metric, the second variation  $$G''=\frac{1}{2}\int_M[(\tri_\vphi+\l)\vphi']^2\om^n\geq 0.$$ The equality holds if and only if $\vphi'$ vanishes.
  \item The gradient flow is of the form $$f'=-f .$$
\end{enumerate}
Now we emphasis on the gradient flow of the $G$-functional: $f'=-f $ with smooth initial K\"ahler potential $\vphi(0)=\vphi_0$. Solving this ODE, we have $f=e^{-t}f_0$ for $f_0=-\log\frac{\om^n_{\vphi_0}}{\om^n}+h_\om-\l\vphi_0$. Now the new equation is a Monge-Amp\`{e}re equation with the right hand side depends on $t$. I.e.
\begin{align*}
\frac{\om^n_\vphi}{\om^n}=e^{h_\om-\l\vphi-e^{-t}f(0)}\doteq e^F.
\end{align*}
The advantage of this flow is that the term $e^{-t}$ is always bounded. The apriori estimats achieve from the following steps.
When $C_1<0$, $C^0$-estimate follows from the maximum principle. At the maximum point $p$, we have $\vphi=-h_\om+\frac{\om^n_\vphi}{\om^n}+e^{-t}f(0)\leq -h_\om+e^{-t}f(0)$ which is bounded from above when $t$ is large. The lower bound follows in the same way. When $C_1=0$, $C^0$-estimate follows from Yau's estimate \cite{Yau}.

The second order estimate also follows from Yau's paper \cite{Yau}. Choosing $C$ such that $C+\inf_{i\neq{k}}R_{i\bar{i}k\bar{k}})=1$,
at the
point $p$ where $e^{-C\vphi}(n+\tri\vphi)$ achieve the
maximal value, we have
\begin{align*}
&0\geq\tri'(e^{-C\varphi}(n+\tri\vphi))(p)\\
&\geq
e^{-C\vphi}\{-Cn(n+\tri\vphi)+(C+\inf_{i\neq{k}}R_{i\bar{i}k\bar{k}})(n+\tri\vphi)^{\frac{n}{n-1}}
e^{\frac{-F}{n-1}}+\tri{F}-n^2(\inf_{i\neq{k}}R_{i\bar{i}k\bar{k}})\}(p).
\end{align*}
Here $\tri F=\tri( h_\om-\l\vphi-e^{-t}f(0))$.
So the second order estimate follows by combining the $C^0$-estimate.

The $C^{2,\a}$-estimate follows from Evans-Krylov estimate, for the concrete constant dependence is specified in Proposition 4.6 in Calamai-Zheng \cite{CalamaiZheng}. So we revisit the existence theorem of K\"ahler-Einstein metric when $C_1 (M) \leq 0$ which was proved by the continuity method of Yau \cite{Yau} and by the K\"ahler-Ricci flow by Cao \cite{MR0799272}.
\begin{thm}\label{exke}
When $C_1 (M) \leq 0$, the gradient flow of the $G$-functional has long time existence and converges to a K\"ahler-Einstein metric.
\end{thm}
In the following we discuss the general properties of the distance function, $K$-energy and the $G$-functional in any K\"ahler class.
\begin{prop}The following geometric inequalities hold.
\begin{enumerate}
  \item The energy inequality: \begin{align}\label{varine}
\nu(\om_\vphi)-\nu(\om)\leq d(\om_\vphi,\om)\sqrt{G(\om_\vphi)}\; .
\end{align}
  \item The geodesic distance has lower bound: \begin{align*}
d(\vphi_0,\vphi_1)\geq \max \left\{
-V^{-\frac{1}{2}}\int_{\tri_{\vphi_0}\vphi'_0<0} \tri_{\vphi_0}\vphi'_0 \om^n_{\vphi_0},\;
V^{-\frac{1}{2}}\int_{\tri_{\vphi_1}\vphi'_1>0} \tri_{\vphi_1}\vphi'_1 \om^n_{\vphi_1} \; .
\right\}\end{align*}
  \item Triangle inequality:
  $d(\vphi_0,\vphi_1)-d(\vphi_1,\vphi_2)\leq d(\vphi_0,\vphi_2)\leq d(\vphi_0,\vphi_1)+d(\vphi_1,\vphi_2)\; .$
\end{enumerate}
\end{prop}
\begin{proof}
Proof of (1). We first connect $\om_\vphi$ and $\om$ by the smooth geodesic. Since
\begin{align*}
\frac{d}{dt}\nu(\om_\vphi)=-\int_M\tri_\vphi\vphi'\frac{\om^n_\vphi}{\om^n}
f\om^n \; ,
\end{align*}
then the proposition follows from the Cauchy-Schwarz inequality
\begin{align*}
\nu(\om_\vphi)-\nu(\om)
&=-\int_0^1\int_M\tri_\vphi\vphi'\frac{\om^n_\vphi}{\om^n}f\om^ndt\\
&\leq d(\om_\vphi,\om)\sqrt{G(\om_\vphi)} \; .
\end{align*}

Proof of (2).
From the geodesic equation
$\frac{d}{dt}\left( \Delta_\vphi \vphi '\right) =- \left(\Delta_\vphi \vphi ' \right)^2 \leq 0$,
we have $$\Delta_{\vphi_1} \vphi '_1\leq \Delta_{\vphi(t)} \vphi '(t)\leq \Delta_{\vphi_0} \vphi '_0 \; .$$
Then the energy $E$ at $t=1$ is bounded by below by the Schwarz inequality,
\begin{align*}
\sqrt{E(1)}\geq V^{-\frac{1}{2}}\int_{\tri_{\vphi_1}\vphi'_1>0} \tri_{\vphi_1}\vphi'_1 \om^n_{\vphi_1} \;.
\end{align*}
Similarly, we have at $t=0$,
\begin{align*}
\sqrt{E(0)}\geq -V^{-\frac{1}{2}}\int_{\tri_{\vphi_0}\vphi'_0<0} \tri_{\vphi_0}\vphi'_0 \om^n_{\vphi_0}\; .
\end{align*}
Since $\p_t E=0$ along the geodesic, we have $\sqrt{E(t)}\geq \max \{\sqrt{E(0)},\sqrt{E(1)}\}.$
Thus the proposition follows since $d=\int_0^1\sqrt{E(t)}dt$.

Proof of (3). The triangle inequality follows from the flat curvature property directly.
\end{proof}
Choosing the path $\vphi(t)=t\vphi$, we have
\begin{cor}
If there is a positive constat $C$ such that $\frac{1}{C}\leq\tri\vphi\leq C$ and $G(\om_\vphi)\leq C$, then the $K$-energy of $\om_\vphi$ is bounded from above by a constant which depends on $C_1$.
\end{cor}
In particular, when $\l$ is zero, the upper bound of $\nu(\om_\vphi)$ only depends on $\frac{1}{C}\leq\tri\vphi\leq C$. When $\l$ is nonzero, it depends on both $\frac{1}{C}\leq\tri\vphi\leq C$ and $\int_M\vphi^2\om^n$.
\begin{cor}
The $K$-energy has lower bound when $M$ admits a constant scalar curvature K\"ahler metric.
\end{cor}
\begin{proof}
In \eqref{varine} choose $\om_\vphi$ to be a constant scalar curvature K\"ahler metric
and $\om$ to be any K\"ahler metric.
Then this properties follows from $G(\om_\vphi)=\nu(\om_\vphi)=0$.
\end{proof}

\subsection{The metric space structure; two non-isometric flat metrics}
 We are going to see that the space $(\mathcal{C}, <\cdot ,\, \cdot >_{F;\, 1})$
 plays as a counterpart of the space
 $(\mathcal{H}, \int_M \psi^2 \frac{\omega^n}{n!} )$.
 This is to say, both the spaces are flat,
 but a pleasant feature is that they are not isometric,
 so that we can interpret $(\mathcal{C}, <\cdot ,\,\cdot >_{F;\, 1})$
 (via the isometric model given by the Monge-Amp\`ere map)
 as another flat structure on the space of K\"ahler metrics
 $\mathcal{H}$.\\
 Now that we observed that the scalar product \eqref{eqn: defn flat metric}
is flat, it is natural to address the following question.
Consider, on the space of K\"ahler metrics $\mathcal{H}$, the
trivial scalar product given by
\begin{align}\label{eqn: trivial flat scalar product}
 <\psi_1 , \psi_2 >_\phi := \int_M \psi_1 \psi_2 \frac{\omega^n}{n!}\; .
\end{align}
The independence of the right hand side on subscript $\phi$
appearing on the left hand side turns the above scalar product
into a flat one.
Moreover, we have just seen that this other scalar product
on $\mathcal{H}$ is flat as well
\begin{align*}
 (\psi_1 , \psi_2 )_\phi =   \int_M (\Delta_{\phi} \psi_1 ) \cdot
 (\Delta_{\phi} \psi_2 )
\frac{\omega_{\phi}^n}{\omega^n}
\frac{\omega_\phi^n}{n!}\; .
\end{align*}
The question is whether the two scalar products are isometric or not.
To answer the question, we want to compare the distance functions induced by the two
flat structures.
In the case of the scalar product \eqref{eqn: trivial flat scalar product},
for any two (normalized) K\"ahler potentials $\phi_0$ and $\phi_1$, there
is a unique  geodesic connecting them, which is given by
\begin{align*}
 \phi(t)= (1-t)\phi_0 + t \phi_1 \; ,
\end{align*}
with $t\in [0,1]$. The fact that this geodesic is contained in $\mathcal{H}$
follows from the Euclidean-convexity of that space.
The distance function induced by that metric is the same as the length
of a geodesic arc. Namely, the distance between $\phi_0$ and $\phi_1$
for this structure, which we label as $d_{\trivial}(\phi_0 , \phi_1)$ is
\begin{align}\label{eqn: distance function for trivial structure}
 d_{\trivial}(\phi_0 , \, \phi_1) = \int_M (\phi_1 - \phi_0)^2 \frac{\omega^n}{n!} \; .
\end{align}
In order to get the distance function for the scalar product \eqref{eqn: defn flat metric},
we look at the expression \eqref{eqn: geodesic arc for the constant weight metric}
in Theorem \ref{thm: geodesics for the constant weight metric}; then, we recall that
under the Monge-Amp\`ere diffeomorphism, there exist two K\"ahler metrics
$\phi_0$ and $\phi_1$ such that $F_k = \frac{\omega_{\phi_k}^n}{\omega^n}$ with $k=0,1$.
Thus, the  expression \eqref{eqn: geodesic arc for the constant weight metric}
of the geodesic arc can be rewritten, in the space $\mathcal{H}$, as
\begin{align*}
 \frac{\omega_{\phi (t)}^n}{\omega^n} =
 (1-t )\frac{\omega_{\phi_0}^n}{\omega^n}
 + t \frac{\omega_{\phi_1}^n}{\omega^n}\; .
\end{align*}
A corollary of Theorem \ref{thm: geodesics for the constant weight metric},
together with the argument of \cite[Lemma 3]{Calamai}, which proves that geodesics are length minimizing,
give the following fact
\begin{prop}
 The constant weight metric endows the space $\mathcal{C}$ with a metric structure.
 The length of a geodesic arc gives the distance between its endpoints.
\end{prop}
We label the distance in this second case as $\dist(\phi_0 ,\,  \phi_1)$,
and we get the expression
\begin{align*}
 \dist(\phi_0 ,\,  \phi_1)
 =
 \int_M
 \left(
 \frac{\omega_{\phi_1}^n - \omega_{\phi_0}^n}{\omega^n}
 \right)^2
 \frac{\omega^n}{n!} \; .
\end{align*}

To summarize the above discussion, we state the following result
\begin{prop}
 The space of K\"ahler metrics $\mathcal{H}$ endowed with the
 Riemannian structure  \eqref{eqn: defn flat metric}
 is not isometric to the same space endowed with the
 Riemann structure \eqref{eqn: trivial flat scalar product}.
\end{prop}

We are going to discuss the metric completion of the constant weight metric.
It is apparent that the constant weight metric is a $L^2$ metric with respect to the
measure induced by the fixed K\"ahler metric $\omega$. Thus, it is immediate to get the following formula
 \begin{prop}
  The metric completion space for the constant weight metric is given by
  \begin{align*}
   \overline{\mathcal{C}}^{L^2(M, \omega)}
   =
   \left\{
   F\in L^2(M, \omega) \; | \; F\geq 0\,  \omega^n - a.e.
   \right\} \; .
  \end{align*}
 \end{prop}

\bigskip


\bigskip
\begin{thebibliography}{10}
\bibitem{MR2040638} C.\ Arezzo, G.\ Tian, \textsl{Infinite geodesic rays in the space of K\"ahler potentials.} (English summary),
Ann. Sc. Norm. Super. Pisa Cl. Sci. (5) 2 (2003), no. 4, 617-630.

\bibitem{MR2491281} J.P.\ Bourguignon \textsl{Ricci curvature and measures.} Jpn. J. Math. 4 (2009), no. 1, 27-45.

\bibitem{Calabi1 } E.\ Calabi, \textsl{ Extremal Kähler metrics.} Seminar on Differential Geometry, pp. 259-290, Ann. of Math. Stud., 102, Princeton Univ. Press, Princeton, N.J., 1982.

\bibitem{Calabi2} E.\ Calabi, \textsl{ Extremal K\"ahler metrics. II.},
Differential geometry and complex analysis, 95-114, Springer, Berlin, 1985.

\bibitem{MR0085583} E.\ Calabi, \textsl{ On K\"ahler manifolds with vanishing canonical class.} Algebraic geometry and topology. A symposium in honor of S. Lefschetz, pp. 78-89. Princeton University Press, Princeton, N. J., 1957.

\bibitem{Calabi} E.\ Calabi, \textsl{The variation of K\"ahler metrics. I. The structure of the space; II. A minimum problem},
Bull. Amer. Math. Soc. 60 (1954), 167-168.

\bibitem{Chen2 }E.\ Calabi, X.X.\ Chen, \textsl{The space of K\"ahler metrics. II.} J. Differential Geom. 61 (2002), no. 2, 173-193.

\bibitem{Calamai} S.\ Calamai, \textsl{The Calabi metric for the space of K\"ahler metrics},
Math. Ann. {\bf 353} (2012) 373-402.

\bibitem{CalamaiZheng} S.\ Calamai and  K. Zheng,
\textsl{Geodesics in the space of K\"ahler cone metrics}, http://arxiv.org/abs/1205.0056.

\bibitem{MR0799272} H.D.\ Cao \textsl{Deformation of K\"ahler metrics to K\"ahler-Einstein metrics on compact K\"ahler manifolds.} Invent. Math. 81 (1985), no. 2, 359-372.

\bibitem{MR1772078} X.X.\ Chen, \textsl{On the lower bound of the Mabuchi energy and its application}, Internat. Math. Res. Notices.   {\bf 12} (2000), 607-623.

\bibitem{Chen} X.X.\ Chen, \textsl{The space of K\"{a}hler metrics}, J.
Differential Geom.   {\bf 61} (2002), 173-193.



\bibitem{chen3} X.X.\ Chen, \textsl{Space of K\"ahler metrics. III. On the lower bound of the Calabi energy and geodesic distance.} Invent. Math. 175 (2009), no. 3, 453-503.

\bibitem{chen4} X.X.\ Chen,
    \textsl{Space of K\"ahler metrics (IV)--On the lower bound of the K-energy}, arXiv:0809.4081

\bibitem{chentang} X.X.\ Chen, Y.D.\ Tang,  \textsl{Test configuration and geodesic rays.} Géométrie différentielle, physique mathématique, mathématiques et société. I. Astérisque No. 321 (2008), 139-167.


\bibitem{ChenTian}
X.X.\ Chen and G.\ Tian , \textsl{Geometry of K\"ahler metrics and
foliations by holomorphic discs}, Publ. Math. Inst. Hautes Etudes
Sci. No. 107 (2008), 1-107.

\bibitem{ChenZheng}
X.X.\ Chen and K.\ Zheng, \textsl{The pseudo-Calabi Flow},
J. Reine Angew. Math. {\bf 674} (2013), 195--251.


\bibitem{ClaRub} B.\ Clarke and Y.A.\ Rubinstein, \textsl{Ricci flow and the completion of the space of K\"ahler metrics}, preprint, arxiv 1102.3787.

\bibitem{ClaRub2} B.\ Clarke and Y.A.\ Rubinstein, \textsl{Conformal deformations of the Ebin metric and a generalized Calabi metric on the space of Riemannian metrics}, to appear in Ann. Inst. H. Poincar\'e Anal. Non Lin\'eaire.

\bibitem{DarvasLempert} T.\ Darvas and L.\ Lempert,
\textsl{Weak geodesics in the space of K\"ahler metrics}, arXiv:1205.0840 .

\bibitem{Don} S.K.\ Donaldson, \textsl{Symmetric spaces, K\"{a}hler geometry and
Hamiltonian dynamics}, In:\ "Northern California Symplectic
Geometry Seminar",\ 13--33,\ Amer.\ Math.\ Soc.\ Transl.\ (2) \
{\bf 196} Amer.Math. Soc., Providence, RI , 1999.

\bibitem{Ebin}   D.G.\ Ebin, \textsl{The manifold of Riemannian metrics}, in: Global analysis (Chern et al., Eds.),
Proceedings of Symposia in Pure and Applied Mathematics (1970), pp.11-40.

\bibitem{MR0718940} A.\ Futaki, \textsl{An obstruction to the existence of Einstein Kähler metrics.} Invent. Math. 73 (1983), no. 3, 437-443.

\bibitem{LempertVivas} L.\ Lempert and L.\ Vivas, \textsl{Geodesics in the space of K\"ahler metrics}, preprint, arXiv:1105.2188.

\bibitem{Mabuchi on K energy}
T.\ Mabuchi, \textsl{$K$-energy maps integrating Futaki invariants}, Tohoku Math. Journ.. \ {\bf 38} (1986),
575-593.

\bibitem{Mab} T.\ Mabuchi, \textsl{Some symplectic geometry on compact
K\"{a}hler manifolds}, Osaka J. Math. \ {\bf 24} (1987),
227--252.

\bibitem{Sem}
S.\ Semmes, \textsl{Complex Monge-Amp\`ere and symplectic
manifolds}, Amer. J. Math., \ {\bf 114} (1999), 495-550.

\bibitem{MR1787650}
G.\ Tian,
\textsl{Canonical metrics in K\"ahler geometry},
Notes taken by Meike Akveld. Lectures in Mathematics ETH ZŸrich. BirkhŠuser Verlag, Basel, (2000), vi+101 pp.

\bibitem{Yau} S. T.\ Yau, \textsl{On the Ricci Curvature of a Compact K\"{a}hler
Manifold and the Complex Monge-Amp\`ere Equation}, Comm. Pure Appl. Math., {\bf 31} (1978), 339-411.

\bibitem{Ast} \textsl{	Premi\`ere classe de Chern et courbure de Ricci:
preuve de la conjecture de Calabi},
Lectures presented at a S\'eminaire held at the Centre de
Math\'ematiques de l'\'Ecole Polytechnique, Palaiseau, March–June 1978.
Astrisque, 58. Soc. Math. France, Paris, 1978. 169 pp.


\end{thebibliography}
\end{document}